\documentclass[11pt]{amsart} 

\usepackage{fullpage}

\usepackage{url}

\usepackage{amssymb}

\usepackage{cite}

\usepackage{graphicx}

\usepackage[usenames]{color}

\usepackage{accents}

\renewcommand{\a}{\alpha}
\renewcommand{\b}{\beta}

\renewcommand{\d}{\delta}
\newcommand{\D}{\Delta}

\newcommand{\cF}{{\mathcal F}}
\newcommand{\cC}{{\mathcal C}}

\newcommand{\cT}{{\mathcal T}}

\newcommand{\cB}{{\mathcal B}}
\newcommand{\cL}{{\mathcal L}}
\newcommand{\cE}{{\mathcal E}}

\newcommand{\cV}{{\mathcal V}}
\newcommand{\cP}{{\mathcal P}}

\newcommand{\cH}{{\mathcal H}}

\newcommand{\cW}{\mathcal W}
\newcommand{\cZ}{\mathcal Z}

\newcommand{\bR}{\mathbb R}

\newcommand{\bP}{\mathbb P}

\newcommand{\be}{\begin{equation}}
\newcommand{\ee}{\end{equation}}

\newcommand{\tr}{\mathrm{tr}}

\newcommand{\beaa}{\begin{eqnarray*}}
\newcommand{\bea}{\begin{eqnarray}}
\newcommand{\beal}[1]{\begin{eqnarray}\label{#1}}
\newcommand{\bean}{\begin{eqnarray}\nonumber}
\newcommand{\beadl}[1]{\begin{deqarr}\label{#1}}
\newcommand{\eeadl}[1]{\arrlabel{#1}\end{deqarr}}
\newcommand{\eeal}[1]{\label{#1}\end{eqnarray}}
\newcommand{\eead}[1]{\end{deqarr}}
\newcommand{\eea}{\end{eqnarray}}
\newcommand{\eeaa}{\end{eqnarray*}}

\newcommand{\p}{\partial}

\renewcommand{\to}{\rightarrow}

\DeclareMathOperator{\vol}{vol}

\renewcommand{\phi}{\varphi}
\renewcommand{\epsilon}{\varepsilon}
\renewcommand{\hat}{\widehat}

\newcommand{\<}{\langle}
\renewcommand{\>}{\rangle}

\newcommand{\w}{\widetilde}

\theoremstyle{plain}
\newtheorem{theorem}{Theorem}[section]

\newtheorem{remark}[theorem]{Remark}

\newtheorem{lemma}[theorem]{Lemma}

\newtheorem{Proposition}[theorem]{Proposition}
\newtheorem{proposition}[theorem]{Proposition}
\newtheorem{corollary}[theorem]{Corollary}

\newtheorem{conjecture}[theorem]{Conjecture}

\theoremstyle{definition}
\newtheorem{definition}[theorem]{Definition}

\def\blacksquare{\hbox to .60em {\vrule width .60em height .60em}}

\numberwithin{equation}{section}

\begin{document}

\title[ ]{On Mass-minimizing Extensions of Bartnik Boundary Data }

\author{Zhongshan An}

\address{Department of Mathematics, 
University of Connecticut,
Storrs, CT 06269}
\email{zhongshan.an@uconn.edu}

\begin{abstract}
We prove that the space of initial data sets which solve the constraint equations and have fixed Bartnik boundary data is a Banach manifold. Moreover if an initial data set on this constraint manifold is a critical point of the ADM total mass, then it must admit a generalised Killing vector field which is asymptotically proportional to the ADM energy-momentum vector. 
\end{abstract}

\maketitle


\section{Introduction}
The Bartnik quasi local mass is one of the most interesting and well-studied notions of quasi local mass in general relativity. For an initial data set  $(\Omega, g_0, K_0)$, which consists of a compact 3-manifold $\Omega$ with nonempty boundary $\partial\Omega$, a Riemannian metric $g_0$ and a symmetric (0,2)-tensor $K_0$ defined on $\Omega$, the Bartnik quasi local mass is defined as (cf.\cite{B1})
\begin{equation}\label{Bmass}
m_{B}(\Omega,g_0,K_0)=\text{inf}\big\{\mbox{the ADM total mass }m_{\text{ADM}}(M,g,K)\big\}.
\end{equation}
Here the infimum is taken over all \textit{admissible extensions} 
\footnote{An extension $(M,g,K)$ is called admissible (cf.\cite{B1}) if it satisfies the dominant energy condition and certain decay conditions so that $m_{\rm ADM}$ is well-defined for the glued data $(M\cup_{\p M}\Omega,g,K)$; in addition $(M\cup_{\p M}\Omega,g,K)$ satisfies certain no-horizon condition. }
$(M,g,K)$ -- asymptotically flat initial data sets such that the boundary $\p M$ can be identified with $\p\Omega$ via some diffeomorphism $\partial M\cong\partial\Omega$, and the following \textit{Bartnik boundary data} of $(M,g,K)$ equals to that of $(\Omega,g_0,K_0)$ along the boundary 
\begin{equation}\label{Bcon}
g_{\partial M}=(g_0)_{\partial\Omega},~H_{\partial M}=H_{\partial\Omega},~\tr_{\partial M}K=\tr_{\partial\Omega}K_0,~\omega_{\partial M}=\omega_{\partial\Omega}.
\end{equation}
In the above, $g_{\partial M}$ is the metric on the boundary induced by $g$; 
$H_{\partial M}$ is the mean curvature of the boundary $\partial M\subset (M,g)$, i.e. $H_{\p M}={\rm div}_g{\bf n}$ with ${\bf n}$ denoting the unit normal vector on the boundary $\p M$ pointing into $M$; 
$\tr_{\partial M}K$ is the trace (with respect to $g_{\p M}$) of the tensor $K|_{\partial M}$ on $\p M$ induced by $K$;
and $\omega_{\partial M}$ is the connection 1-form on $\p M$ defined by $\omega_{\partial M}=K(\mathbf n)|_{\p M}$, i.e. the normal-tangential components $K$ on the boundary. The Bartnik boundary data 
$((g_0)_{\partial\Omega},H_{\partial\Omega},\tr_{\partial\Omega}K_0,\omega_{\partial\Omega})$ of the compact initial data set has the same meaning, except that $H_{\p\Omega}$ and $\omega_{\p\Omega}$ are defined with respect to the unit normal vector on the boundary $\p\Omega$ pointing out of $\Omega$.

Various geometric conditions across the boundary $\p M$ have been studied in the literature, such as $(M,g,K)$ extends $(\Omega,g_0,K_0)$ smoothly or the mean curvature is non-increasing ($H_{\p M}\leq H_{\p \Omega}$). 
The Bartnik boundary condition \eqref{Bcon} is natural in various aspects. On one hand, it ensures that the Hamiltonian constraint $u$ and momentum constraint $Z$ (cf. \eqref{cons} below) can be distributionally well-defined for the glued initial data set $(M\cup_{\p M}\Omega, g, K)$.
In this way, it is reasonable to impose the dominant energy condition ($u\geq |Z|$) and expect that the resulting total ADM mass of the complete manifold $M\cup_{\p M}\Omega$ and hence also the Bartnik quasi local mass of $\Omega$ are non-negative based on the positive mass theorem (cf. \cite{SY}).
Moreover, the Bartnik boundary data also arises naturally from a Hamiltonian analysis of the vacuum Einstein equations, which we will see in the discussion to follow. 

A well-known conjecture on the Bartnik quasi local mass proposed by Bartnik is:
\begin{conjecture}
If the infimum in \eqref{Bmass} is achieved, it must be realized by a stationary vacuum extension -- an extension $(M,g,K)$ which can be embedded into a stationary vacuum spacetime as an initial data set. 
\end{conjecture}
Here a {\it stationary vacuum spacetime} is a spacetime equipped with a Lorentzian metric which is Ricci flat and admits a Killing vector field that is asymptotically time-like. 
We note that in the original conjecture in \cite{B1} stationary spacetimes refer to those admitting time-like Killing vectors.
However, this is a strong condition and hard to prove.
Based on the result in Corollary 6.2 of \cite{B3} and Theorem 1.3 below, we expect the conjecture holds for the more general definition of stationary spacetimes.
Besides, we note that a well-known model of stationary spacetime is the Kerr metric (cf.\cite{W1}). It admits a Killing vector field that is only time-like outside a compact subset. 

The Conjecture 1.1 was first studied in the time-symmetric case 
where $K_0\equiv0$ so that the Bartnik boundary condition \eqref{Bcon} is reduced to
 \begin{equation}\label{rBcon}
\big(g_{\partial M}, H_{\partial M}\big)=\big((g_0)_{\partial\Omega}, H_{\partial\Omega}\big).\end{equation}
Corvino (cf.\cite{C}) proved that if $(M,g)$ is a minimal ADM energy extension which extends $(\Omega,g_0)$ smoothly, then it must be {\it static} in the sense that $g$ admits a nontrivial static potential on $M\setminus\p M$.
In addition, Miao (cf.\cite{Mi}) proved when $\p M$ has positive Gauss curvature, {any} minimal mass extension for the Bartnik quasi local mass, which is defined with non-increasing mean curvature boundary condition, must satisfy condition \eqref{rBcon} as well as being static.
For the general case where the spacetime is not necessarily time-symmetric, Corvino (cf.\cite{C1}) studied it using a modified constraint map and conformal argument.
With further application of the modified constraint map,  Huang-Lee (cf.\cite{HL}) proved that {any} mass minimizer can be embedded into a null dust spacetime which admits a global Killing vector field.

Besides the approaches mentioned above, Bartnik (cf.\cite{B3}) constructed a regularization $\mathcal H$ of the Regge-Teitelboim Hamiltonian and analyzed the functional $\mathcal H$ following an approach initiated by Brill-Deser-Fadeev (cf.\cite{BDF}). By that he proved on a complete asymptotically flat manifold constrained critical points of the ADM total mass must be stationary. Bartnik then suggested that a variational proof of the conjecture, based on extending his work to manifolds with nonempty boundary, would be more natural. By implementing the program suggested by Bartnik, Anderson-Jauregui (cf.\cite{AJ}) proved the conjecture in the time-symmetric case; moreover, they showed that the static potential function of the mass-mininizing metric must be positive and asymptotically decays to 1 at infinity, which has not been addressed by previous work. In this paper, we will generalize the method in \cite{B3,AJ} to study the mass-minimizing extensions for compact initial data sets in general and extend the result on critical points of the ADM total mass in \cite{B3} to asymptotically flat manifold with nonempty interior boundary where the Bartnik boundary data is fixed. This will further prove part of Conjecture 1.1.

Recall that the Einstein equation for a spacetime $(V^{(4)},g^{(4)})$ is given by
\begin{equation}\label{vac}
Ric_{g^{(4)}}-\frac{1}{2}R_{g^{(4)}}=8\pi T,\end{equation}
where $g^{(4)}$ is a metric with signature $(-,+,+,+)$; $Ric_{g^{(4)}}$ and $R_{g^{(4)}}$ are the Ricci curvature and scalar curvature of $g^{(4)}$; and $T$ is the stress-energy tensor of matter. If an initial data set $(M,g,K)$ is embedded in such a spacetime, it must satisfy the \textit{constraint equations}
\begin{equation}\label{cons}
\begin{cases}
R-|K|^2+(\tr K)^2=u,\\
{\rm div} K-d\tr K=Z,
\end{cases}
\end{equation}
where the norm $|\cdot|^2$, trace $\tr$ and divergence ${\rm div}$ operators are all with respect to the metric $g$. The first equation above is called the {\it Hamiltonian constraint} and the second is called the {\it momentum constraint}. The constraint equations are obtained by decomposing the Einstein equation \eqref{vac} according to the Gauss-Codazzi-Mainardi hypersurface equations on $M\subset (V^{(4)},g^{(4)})$.

Consider the {\it constraint space} $\mathcal C(u,Z)$ of all asymptotically flat initial data sets $(M,g,K)$ which satisfy the constraint equations \eqref{cons} for some fixed $(u,Z)$.
For complete asymptotically flat manifolds $M$, Bartnik proved (cf.\cite{B3}) the constraint space $\mathcal C(u,Z)$ has Hilbert manifold structure.
In a recent work \cite{Mc1} McCormick generalized this result to asymptotically flat manifolds with interior boundary where there is no boundary condition.
In view of the Bartnik quasi local mass, it is of great interest to examine the space $\cC(u,Z)$ subjected to certain boundary conditions.
However, a crucial ingredient in Bartnik's proof -- surjectivity of the constraint map -- becomes complicated and subtle when we impose geometric boundary conditions. In fact, McCormick (cf.\cite{Mc}) worked with the boundary condition which requires the first derivatives of the metric to be fixed on $\partial M$ and pointed out that the manifold structure theorem is almost certainly false in this case. In \cite{AJ} Anderson-Jauregui proved the manifold structure for the constraint space in the time-symmetric case, where the ellipticity of the boundary data \eqref{rBcon} for static spacetimes (cf.\cite{AK}) plays an important role.  

Inspired by the work \cite{AJ}, we apply ellipticity of the Bartnik boundary data to prove that the constraint space $\mathcal C(u,Z)$ admits a Banach manifold structure when the Bartnik boundary data \eqref{Bcon} is fixed. 
In this paper, we work with manifolds $M$ which are diffeomorphic to the exterior region $\mathbb R^3\setminus B^3$, where $B^3$ denotes the open unit 3-ball, and use weighted H\"older ($C^{m,\alpha}_{\delta}$) norm to control the asymptotic behavior of tensor fields on $M$.
Roughly speaking, a tensor field has bounded $C^{m,\alpha}_{\delta}$-norm if it is $C^{m,\alpha}$ smooth and decays to zero at the rate $r^{-\d}$ where $r$ is the radius function on $\mathbb R^3\setminus B$. We refer to Definition 2.1 for the precise definition.
Throughout the paper, we consider asymptotically flat initial data $(g,K)$ on $M$ where $g$ decays to the flat metric in the $C^{m,\alpha}_\d$-norm and $K$ has bounded $C^{m-1,\alpha}_{\d+1}$-norm for $m\geq 2,0<\alpha<1,1/2<\d\leq 1$. 
One of the main results of this paper is 
\begin{theorem}
Given $(u,Z)$ with bounded $C^{m-2,\alpha}_{\d+2}$-norm, the constraint space $\mathcal C_B(u,Z)$, which consists of asymptotically flat initial data $(g,K)$ on $M$ satisfying \eqref{cons} and \eqref{Bcon}, is an infinite-dimensional smooth Banach manifold. 
\end{theorem}
This theorem is proved by applying the implicit function theorem for Banach spaces to the constraint map. In \S2, we construct a constraint map $\Phi$ based on \eqref{cons} and the boundary conditions \eqref{Bcon} so that the constraint space $\mathcal C_B(u,Z)$ is equal to the level set of $\Phi$. 
We will show this constraint map is a submersion, i.e. its linearization is surjective and has splitting kernel.
In \S2.1 we prove the linearization $D\Phi$ is surjective by showing it has closed range and trivial cokernel, where the closed-range property is essentially due to the ellipticity of the Bartnik boundary data for stationary vacuum spacetimes.
A detailed discussion of the ellipticity is given in \S2.3.
In addition, we prove that the linearized constraint map has splitting kernel in \S2.2 using the idea developed in \cite{W}. 

We refer to $\cC_B(u,Z)$ as the {\it constraint manifold} since it admits Banach manifold structure.
On this constraint manifold, we consider the modified Regge-Teitelboim Hamiltonian $\cH$ constructed in \cite{B3}. In \S3 we analyze the variational formula of $\cH$ and show that the boundary terms in the formula vanish for infinitesimal deformations preserving the Bartnik boundary data. So it follows that the variational problem for the Hamiltonian $\cH$ on the constraint manifold is well-defined. Then we study critical points of the ADM total mass on the constraint manifold, following the approach suggested by Bartnik in \cite{B3}. A rough version of the main theorem is as follows, we refer to Theorem 3.2 for a precise statement.
\begin{theorem}
For $(u,Z)$ with bounded $C^{k,\a}_{q}$-norm $(k\geq 0,q\geq 4)$, critical points of the ADM total mass on the constraint manifold $\mathcal C_B(u,Z)$ which have positive ADM total mass are exactly the initial data sets admitting generalised Killing fields that are asymptotically time-like. 
\end{theorem}
Here we adopt the terminology {\it generalised Killing vector fields} from the work of Bartnik \cite{B3} -- it refers to nontrivial kernel elements of the adjoint $D\Phi^*$ of the linearized constraint map (cf.\S3 for the precise definition).
Such kernel elements are also called {\it Killing Initial Data} (KIDs) in the literature (cf.\cite{BCS}), which is motived by the following well-known result from \cite{Mo} (cf. also \cite{FM}):
\smallskip\\
{\bf Theorem}
 \footnote{Moncrief worked with vacuum spacetimes with compact Cauchy hypersurfaces in order to discuss the linearization stability of the Einstein equations. But this particular result also holds in noncompact case.}
 (Moncrief) {\it Suppose $(M,g,K)$ is embedded as a Cauchy surface in a smooth globally hyperbolic spacetime $(V^{(4)},g^{(4)})$ which satisfies the vacuum Einstein equation \eqref{vac} with $T=0$. Then a generalised Killing vector field $(X^0,X^i)$ of $(M,g,K)$ gives rise to a standard Killing vector field $X^{(4)}$ in $(V^{(4)},g^{(4)})$ such that the perpendicular and parallel components of $X^{(4)}$ are $X^0$ and $X^i$ on $M$.}\smallskip

Now combining this theorem and Theorem 1.3, we can prove Conjecture 1.1 partially.
In the following we assume that the no-horizon condition in the definition of admissible extensions for Bartnik quasi local mass \eqref{Bmass} is an open condition (cf.\cite{HL} for a discussion of open no-horizon conditions).
Suppose an initial data set $(M,g,K)$ is a mass-minimizing extension for $(\Omega,g_0,K_0)$, where $\p\Omega=S^2$ and $M\cong\mathbb R^3\setminus B$.
Then the initial data $(g,K)$ must be a critical point of the ADM total mass on the constraint manifold $\cC_B(u,Z)$ that contains it.
Assume in addition the infimum \eqref{Bmass} is positive, then by Theorem 1.3 $(M,g,K)$ must admit a generalised Killing vector field which is asymptotically time-like.
If in addition this initial data set satisfies the vacuum constraint equations \eqref{cons} with $u=Z=0$, then one can construct a vacuum spacetime $(V^{(4)},g^{(4)})$ where $(M,g,K)$ is embedded as a Cauchy surface. 
It then follows from the Theorem by Moncrief that the vacuum spacetime $(V^{(4)},g^{(4)})$ is stationary, i.e. it admits a Killing vector field that is asymptotically time-like. 
This leads to the following corollary:
\begin{corollary}
If $(M,g,K)$ is a smooth asymptotically flat initial data set realizing the infimum in \eqref{Bmass} with positive ADM total mass and satisfying the vacuum constraint equations, then it must arise from a vacuum stationary spacetime. 
\end{corollary}  
\begin{remark}
{\rm The ambient vacuum spacetime $(V^{(4)},g^{(4)})$ can be constructed as a solution to the Cauchy problem of the vacuum Einstein equations. For this Cauchy problem to be well-posed, the initial data $(g,K)$ must have enough regularity. In the original work of Choquet-Bruhat (\cite{CB}), the initial data must be $C^5\times C^4$-smooth to guarantee existence of solutions locally in time. We refer to \cite{HR} for a detailed discussion and recent development on the regularity issue. 
In the corollary above we assume the initial data set is smooth for simplicity. However, with respect to the weighted H\"older norm used in this paper, the result should work well for initial data sets $(M,g,K)$ where $(g,K)\in Met^{m,\alpha}_{\d}\times S^{m-1,\alpha}_{\d+1}$ for $m\geq 5$. 
}
\end{remark}
It remains an open and interesting problem whether a minimizer of the Bartnik mass must belong to the vacuum constraint manifold $\cC_B(0,0)$. We note that in a recent work by Huang-Lee \cite{HL} the constraints of a minimizer is well-studied and in particular they are shown to be vacuum outside a compact set of $M$.

\medskip
\textbf{Acknowledgements}
I would like to express great thanks to my Ph.D advisor Prof. Michael Anderson for suggesting this problem, and thanks to Prof. Michael Anderson and Prof. Lan-Hsuan Huang for valuable discussions and comments.

\section{The Constraint Manifold}
Let $M$ be a smooth manifold with nonempty boundary diffeomorphic to $\mathbb R^3\setminus B^3$, where $B^3$ denotes the open unit 3-ball. So its boundary $\p M$ is diffeomorphic to the unit sphere $S^2$. Via the diffeomorphism $M\cong\mathbb R^3\setminus B$, $M$ can be equipped with a global coordinate chart $\{x^i\},(i=1,2,3)$, a radius function $r\in [1,\infty)$ and a flat metric $\mathring g$ which is the pull back of the flat metric on $\mathbb R^3\setminus B^3$. Using this chart, we can define the weighted H\"older spaces of tensor fields on $M$ as follows.
\begin{definition}
Let $m$ be a nonnegative integer, $\alpha\in(0,1)$ and $\d\in\mathbb R$. The $C^m_{\delta}$-norm of a $C^m$ function $v$ on $M$ is given by
$$||v||_{C^m_{\delta}}=\textstyle\sum_{k=0}^m{\rm sup}~ r^{k+\delta}|\mathring{\nabla}^kv|$$
where $\mathring{\nabla}$ is the connection with respect to $\mathring{g}$.
The $C^{m,\alpha}_{\delta}$-norm of a $C^{m,\alpha}$ function $v$ on $M$ is given by 
$$||v||_{C^m_{\delta}}+{\rm sup}_{x\neq y}\{{\rm min}( r(x), r(y))^{m+\alpha+\delta}
\frac{|\mathring{\nabla}^mv(x)-\mathring{\nabla}^mv(y)|}{|x-y|^{\alpha}}\}.
$$
The space $C^m_{\delta}(M)$ (or $C^{m,\alpha}_{\delta}(M)$) is the space of all functions with bounded $C^m_{\delta}$-norm ( or $C^{m,\alpha}_{\delta}$-norm). Various spaces of tensor fields on $M$ with respect to the weighted H\"older norm are defined as
\begin{equation*}
    \begin{split}
&Met^{m,\alpha}_{\delta}(M)=\{\text{Riemannian metrics}~g~\text{on}~M: (g_{ij}-\mathring g_{ij})\in C^{m,\alpha}_{\delta}(M)\},\\
&S^{m,\alpha}_{\delta}(M)=\{\text{symmetric (0,2)-tensors }~K~\text{on}~M: K_{ij}\in C^{m,\alpha}_{\delta}(M)\},\\
&T^{m,\alpha}_{\delta}(M)=\{\text{vector fields }~Y~\text{on}~M: Y^i\in C^{m,\alpha}_{\delta}(M)\},\\
&(T_p^q)^{m,\alpha}_{\delta}(M)=\{(q,p)-\text{tensors}~\tau~\text{on}~M: \tau_{i_1i_2..i_p}^{j_1j_2...j_q}\in C^{m,\alpha}_{\delta}(M)\},\\
&(\wedge_p)^{m,\alpha}_{\delta}(M)=\{p-\text{forms}~\sigma~\text{on}~M: \sigma_{i_1i_2..i_p}\in C^{m,\alpha}_{\delta}(M)\}.
    \end{split}
\end{equation*}
\end{definition}
On $M$, an asymptotically flat initial data set consists of a Rimannian metric $g\in Met^{m,\alpha}_{\delta}(M)$ and a symmetric 2-tensor $K\in S^{m-1,\alpha}_{\delta+1}(M)$. 
Throughout, we assume that $m\geq 2$ and $\frac{1}{2}<\delta< 1$.
Based on the Bartnik boundary condition (1.2), we set up a space  $\mathcal B$ of initial data on $M$ with fixed Bartnik boundary data: 
\begin{equation*}
\begin{split}
\mathcal B=\{&(g,K)\in [Met_{\delta}^{m,\alpha}\times S_{\delta+1}^{m-1,\alpha}](M):\\
& (g_{\partial M},~H_{\partial M},~\tr_{\partial M}K,~\omega_{\partial M})=~\big((g_0)_{\partial\Omega},H_{\partial\Omega}, \tr_{\partial\Omega}K_0,\omega_{\partial\Omega}\big)\text{ on }\partial M \},
\end{split}
\end{equation*}
where $(\Omega,g_0,K_0)$ is a fixed compact initial data set with boundary $\p\Omega\cong S^2$.
It is easy to show by implicit function theorem that for a fixed set of data $\big((g_0)_{\partial\Omega},H_{\partial\Omega}, \tr_{\partial\Omega}K_0,\omega_{\partial\Omega}\big)$, $\mathcal B$ is a smooth closed Banach submanifold of $[Met_{\delta}^{m,\alpha}\times S_{\delta+1}^{m-1,\alpha}](M)$. The tangent space at a point $(g,K)\in\cB$ consists of infinitesimal deformations which preserve the Bartnik boundary data, i.e. 
\begin{equation}\label{lBcon}
\begin{split}
T\mathcal B|_{(g,K)}=\{
(&h,p)\in [S_{\delta}^{m,\alpha}\times S_{\delta+1}^{m-1,\alpha}](M):\\
& h^T=0,~H'_h=0,~\tr^Tp=0,~p(\mathbf n)^T+K(\mathbf n'_h)^T=0 \text{ on }\partial M
\}.
\end{split}
\end{equation}
Throughout the paper the superscript $^T$ on a tensor field denotes its components tangential to the boundary manifold $\partial M$. In addition, we use $\tr^T$ to denote the trace of an induced tensor on the boundary manifold with respect to the induced metric $g^T$. The prime $'$ denotes the variation of a geometric tensor with respect to the infinitesimal deformation $h$ or $p$. More precisely, if $g(t)=g+th$ is a family of metrics on $M$ then $H'_h=\tfrac{d}{dt}|_{t=0}H_{g(t)}$ and ${\bf n}'_h=\tfrac{d}{dt}|_{t=0}{\bf n}_{g(t)}$. Recall that ${\bf n}_g$ denotes the unit normal to the boundary $\p M\subset (M,g)$ pointing inwards and $H_g$ is the mean curvature $H_g={\rm div}_g{\bf n}_g$
\footnote{We extend ${\bf n}$ naturally to a vector field defined in a collar neighborhood of $\p M$ such that $\nabla_{\bf n}{\bf n}=0$}.
In the following we omit the subscript $g$ for simplicity.
Basic calculation shows 
(cf. for example \cite{Az2})
\begin{equation}\label{lhn}
H'_h=\tfrac{1}{2}{\bf n}(\tr^Th)+\d^T(h({\bf n})^T)-\tfrac{1}{2}h({\bf n},{\bf n})H,\ \ {\bf n}'_h=-\tfrac{1}{2}h({\bf n},{\bf n}){\bf n}-h({\bf n})^T.
\end{equation}
Here $\d$ is the negative divergence operator, i.e. $\d \tau=-\tr \nabla\tau=-g^{ab}\nabla_a\tau_{bi_1i_2...i_{p-1}}$ for any $(0,p)$ tensor $\tau$; and $\d^T$ denotes such operator on $\p M$ with respect to the induced metric $g^T$.

Define the constraint map $\Phi$ on $\mathcal B$ as
\begin{equation}\label{consm}
\begin{split}
&\Phi:\mathcal B\to\mathcal T\\
\Phi(g,K)=&\big(\Phi_0(g,K),\Phi_i(g,K)\big),
\end{split}
\end{equation}
where
\begin{equation*}
\begin{split}
&\Phi_0(g,K)=\big(R-|K|^2+(\tr K)^2\big)\sqrt{g},\\
&\Phi_i(g,K)=-2\big(\delta K+d(\tr K)\big)\sqrt{g},
\end{split}
\end{equation*}
with $\sqrt{g}=\sqrt{\det g}/\sqrt{\det\mathring g}$. 
The target space of $\Phi$ is 
$\mathcal T=C^{m-2,\alpha}_{\delta+2}(M)\times (\wedge_1)^{m-2,\alpha}_{\delta+2}(M)$ since $\Phi_0(g,K)\in C^{m-2,\alpha}_{\delta+2}(M)$ is a scalar field and $\Phi_i(g,K)\in (\wedge_1)^{m-2,\alpha}_{\delta+2}(M)$ is a 1-form. By basic computation, the linearization of $\Phi$ at a point $(g,K)\in\cB$ with $\Phi(g,K)=(u,Z)$ is given by
\begin{equation}\label{lconsm}
\begin{split}
&D\Phi_{(g,K)}:T\mathcal B|_{(g,K)}\to\mathcal T\\
&D\Phi_{(g,K)}(h,p)=
\big((D\Phi_0)_{(g,K)}(h,p),(D\Phi_i)_{(g,K)}(h,p)\big),
\end{split}
\end{equation} 
where
\begin{equation}\label{lu}
\begin{split}
(D\Phi_0)_{(g,K)}(h,p)=&R'_h\sqrt{g}+\big(2K_{ik}K_{j}^kh^{ij}-2(\tr K)\langle K,h \rangle\big)\sqrt{g}\\
&-2(\langle K,p\rangle-(\tr K)(\tr p))\sqrt{g}+\tfrac{1}{2}(\tr h)u,
\end{split}
\end{equation} 
\begin{equation}\label{lz}
\begin{split}
(D\Phi_i)_{(g,K)}(h,p)&=-2\big(\delta p+d\tr p\big)\sqrt{g}-2\big(\delta'_hK-d\langle K,h \rangle\big)\sqrt{g}+\tfrac{1}{2}(\tr h)Z.
\end{split}
\end{equation} 
In the formulas above variation of the scalar curvature is $R'_h=\D(\tr h)+\d\d h-\<Ric_g,h\>$; and variation of the divergence operator is $(\d'_hK)_i=h^{jk}\nabla_j K_{ki}-K(\b h)_i+\tfrac{1}{2}K_{jk}\nabla_i h^{jk}$ where $\b$ denotes the Bianchi operator $\b h=\d h+\tfrac{1}{2}d\tr h$ (cf.\cite{Bess}). Here and throughout the paper the Laplacian $\D=-\tr{\rm Hess}$.

In this section we will prove for a fixed pair $(u,Z)\in\cT$ the level set $\Phi^{-1}(u,Z)$ is a Banach manifold based on the implicit function theorem. 
Before starting the proof, we note that there is an equivalent way to express the constraint map. Let $\pi$ be the conjugate momentum defined as
\begin{equation*}
\begin{split}
\pi=\big(K-(\tr_gK)g\big)^{\sharp}\sqrt{g}.
\end{split}
\end{equation*}
Here the superscript $^{\sharp}$ means to raise the indices of a (0,2)-tensor with respect to the metric $g$. Let $\mathcal{\w B}$ be the space of pairs $(g,\pi)$ parameterised by $(g,K)$ in $\mathcal B$:
\begin{equation*}
\begin{split}
\mathcal{\w B}=\{(g,\pi)\in [Met^{m,\a}_{\d}\times (T_0^2)^{m-1,\a}_{\d+1}](M): g=g, \pi=\big(K-(\tr_{g}K)g\big)^{\sharp}\sqrt{g}, \text{ for some }(g,K)\in\mathcal B\}.
\end{split}
\end{equation*}
It is easy to observe that the space $\mathcal B$ and $\mathcal{\w B}$ are equivalent. So $\mathcal {\w B}$ is also a Banach manifold. The tangent space at $(g,\pi)\in\w\cB$ is given by
\begin{equation}\label{tB}
\begin{split}
T\mathcal {\w B}|_{(g,\pi)}=\{
&(h,\sigma)\in [S_{\delta}^{m,\alpha}\times (T_0^2)_{\delta+1}^{m-1,\alpha}](M):~\sigma\text{ is a symmetric (2,0)-tensor, }\\
&h^T=0,~H'_h=0,~\sigma^{11}+\frac{1}{2}\pi^{11}h_{11}=0,~\sigma^{1A}+\pi^{11}h_{1A}=0~(A=2,3)
\text{ on } \partial M
\}.
\end{split}
\end{equation}
Here and throughout the paper, we use the index $1$ to denote normal direction to the boundary $\p M\subset (M,g)$ and indices $2,3$ to denote the tangential directions to the boundary $\partial M$. Upper case Roman indicies $A\in\{2,3\}$ and lower case Roman $i\in\{1,2,3\}$. In addition we use index $0$ to denote the time direction in the ambient spacetime which contains the initial data set $(M,g,K)$ and use Greek letters $\mu\in\{0,1,2,3\}$ when needed. 

The boundary conditions in \eqref{tB} are equivalent to those listed in \eqref{lBcon}; we refer to the appendix \S 4.1 for the detailed verification. The constraint map then can be equivalently defined as
\begin{equation}\label{tconsm}
\begin{split}
&\w\Phi:\mathcal {\w B}\to\mathcal T\\
\w\Phi(g,\pi)=&\big(\w\Phi_0(g,\pi),\w\Phi_i(g,\pi)\big),
\end{split}
\end{equation}
with
\begin{equation*}
\begin{split}
\w\Phi_0(g,\pi)=R(g)\sqrt{g}-\big(|\pi|^2-\tfrac{1}{2}(\tr\pi)^2\big)/\sqrt{g},\ \ \w\Phi_i(g,\pi)=-2\big(\delta(\pi/\sqrt{g})\big)^{\flat}\sqrt{g}.
\end{split}
\end{equation*}
Here the superscript $\flat$ means to lower the indices of a tensor field with respect to the metric $g$. We refer to \cite{B3} for the explicit formula of the linearization $D\w\Phi$.
Obviously, the maps $\Phi$ and $\w \Phi$ are related by the equivalence between $\mathcal B$ and $\mathcal{\w B}$, so their level sets are diffeomorphic. In the next section (\S3), we will switch between these two formulations as needed.

\medskip

Now we give the proof that the constraint map $\Phi$ is a submersion, i.e. at a point $(g,K)\in\Phi^{-1}(u,Z)$ the linearized map $D\Phi_{(g,K)}$ given in \eqref{lconsm} is surjective and its kernel splits in $T\mathcal B|_{(g,K)}$, so that we can apply the implicit function theorem on $\Phi$. In the following we will use $\cL$ to denote the linearized constraint map $D\Phi_{(g,K)}$ for simplicity.

\subsection{Surjectivity} We will prove $\cL$ is surjective by showing it has closed range and trivial cokernel.
To prove $\cL$ has closed range, we will construct a subspace $V$ of the tangent space $T\mathcal B|_{(g,K)}$ so that the image $\cL(V)$ has finite codimension in $\mathcal T$. 

For $(g,K)$ define a space $\mathcal W$ consisting of triples $(h,Y, v)$ of a symmetric 2-tensor $h$, a vector field $Y$ and a scalar field $v$, all of which are asymptotically zero on $M$ as follows
\begin{equation}\label{CW}
\begin{split}
\mathcal W=\{&(h,Y, v)\in [S^{m,\alpha}_{\delta}\times T^{m,\alpha}_{\delta}\times C^{m,\alpha}_{\delta}](M):\\
&~\delta h-3dv=0,~h^T=0,~H'_h=0,~ \tr^T\big(\delta^*Y+(\delta Y)g\big)=0,~\delta^*Y(\mathbf n)^T+K(\mathbf n'_h)^T=0
\text{ on }\partial M~
\}.
\end{split}
\end{equation}
Here and throughout the following, the divergence operator $\d$ and its adjoint $\d^*$ are both with respect to the metric $g$. When acting on vector fields, the adjoint of $\d$ is given by the Lie derivative: $\d^*Y=\tfrac{1}{2}L_Yg$.
Variation of the mean curvature $H'_h$ and unit normal ${\bf n}'_h$ are as in \eqref{lhn}, both taken at the base point $g$. All the boundary conditions above are constructed based on the linearized Bartnik boundary conditions listed in \eqref{lBcon}, except that the first boundary condition is regarded as a gauge condition. 

Let $\mathcal V$ be the space obtained from projecting $\mathcal W$ to the first two components, i.e.
\begin{equation}\label{CV}
\begin{split}
\mathcal V=\{(h,Y)\in [&S^{m,\alpha}_{\delta}\times T^{m,\alpha}_{\delta}](M):  (h,Y,v)\in\mathcal W \text{ for some } v\}.
\end{split}
\end{equation}
Then it is easy to verify that the space $V$ generated by $\mathcal V$:
\begin{equation*}
\begin{split}
V=\{(h,p)\in [&S^{m,\alpha}_{\delta}\times S^{m-1,\alpha}_{\delta+1}](M):~
h=h,~p=\delta^*Y+(\delta Y)g\text{ for some }(h,Y)\in\mathcal{V}
\},
\end{split}
\end{equation*}
is a subspace of $T\mathcal B|_{(g,K)}$.
Let $\Psi$ be the map on $\mathcal V$ induced by the linearized constraint map $\cL$ on $V$, i.e.
\begin{equation}\label{PS1}
\begin{split}
&\Psi: \mathcal{V}\to\mathcal T\\
\Psi(h,Y&)=\cL(h,\delta^*Y+(\delta Y)g).
\end{split}
\end{equation}
Based on the formula \eqref{lu}-\eqref{lz}, $\Psi$ is of the form
\begin{equation}\label{PS2}
\begin{split}
\Psi(h,Y)
=\big(~ (\D(\tr h)+\d\d h)\sqrt{g}+O_1(h,Y),~-2[\delta\delta^*Y+d\delta Y]\sqrt{g}+O_1(h,Y)~\big)
\end{split}
\end{equation}
where $O_1(h,Y)$ denote the terms that involve at most 1st order derivatives of $h$ and $Y$. 

Observe the range of $\Psi$ satisfies ${\rm Im}\Psi=\cL(V)$. Since $V$ is a subspace of $T\cB|_{(g,K)}$, the closedness of the range of $\cL$ will hold if we show ${\rm Im}\Psi$ has finite codimension in $\mathcal T$. As mentioned in the introduction, the constraint equations are actually part of the Einstein field equations on the spacetime $(V^{(4)},g^{(4)})$ where $(M,g,K)$ is embedded as an initial data set. So its linearization $\cL$ is part of the linearized Einstein equations. Moreover, it is proved in \cite{Az2} that the stationary Einstein equations (combined with proper gauge) together with the Bartnik boundary conditions form an elliptic boundary value problem in the phase space consisting of triples $(g,X,N)$ on $M$. Here $X$ and $N$ are understood as the shift vector and lapse function on the hypersurface $(M,g)\subset (V^{(4)},g^{(4)})$. So we can understand $(h,Y,v)\in\mathcal W$ defined above as the deformation of $(g,X,N)$ which preserves the Bartnik boundary data; and the map $\Psi$ can be taken as part of the linearized stationary Einstein field equations, which indicates that the map $\Psi$ is underdetermined elliptic. 

To carry out this idea, we first construct a differential operator $P=(L,B)$, with $L$ being the interior operator 
\begin{equation}\label{BVPL}
\begin{split}
&L:[S^{m,\alpha}_{\delta}\times T^{m,\alpha}_{\delta}\times C^{m,\alpha}_{\delta}](M)\to [S^{m-2,\alpha}_{\delta+2}\times C^{m-2,\alpha}_{\delta+2}\times (\wedge_1)^{m-2,\alpha}_{\delta+2}](M)\\
&\quad\quad\quad\quad L(h,Y,v)=
\big(
\cE_0(h,v),~
\Delta \tr h+\delta\delta h,~
\delta\delta^*Y+d\delta Y
\big)
\end{split}
\end{equation}
and $B$ the boundary operator 
\begin{equation}\label{BVPB}
\begin{split}
B:[S^{m,\alpha}_{\delta}\times &T^{m,\alpha}_{\delta}\times C^{m,\alpha}_{\delta}](M)\to [(\wedge_1)^{m-1,\alpha}\times S^{m,\alpha}\times (C^{m,\a})^2\times (\wedge_1)^{m-1,\alpha}](\p M)\\
&B(h,Y,v)=\big(
\delta h-3d v,~
h^T,~
H'_{h},~
\tr^T[\delta^*Y+(\delta Y)g],~
\delta^*Y(\mathbf{n})^T
\big),
\end{split}
\end{equation}
where the first component in $L(h,Y,v)$ is given by
\be\label{CE}
\cE_0(h,v)=Ein'_{h}+\delta^*\delta h-(\delta\delta h)g-4D^2 v+2(\Delta  v)g.
\ee
If we use $^{(4)}Ein$ to denote the Einstein tensor of the spacetime $(V^{(4)},g^{(4)})$, then the interior operator $L$ can be interpreted as mapping $(h,Y,v)$ to the principal part of the linearization $^{(4)}Ein'_{(h,Y,v)}$ combined with an extra term $[\delta^*\delta h-(\delta\delta h)g-4D^2 v+2(\Delta  v)g]$ which serves as a gauge term. Note that the choice of the gauge term is not unique and the one we choose here is for simplicity of the proof of ellipticity to follow. Observe the boundary operator $B$ maps $(h,Y,v)$ to the leading order part of the conditions listed in \eqref{CW}. Using the method developed in \cite{AK}, one can prove that $P$ is an elliptic operator. We refer to \S2.3 for the detailed proof of ellipticity. 

Notice that the leading order terms in formula \eqref{PS2} of $\Psi$ differ from the 2nd and 3rd bulk terms in \eqref{BVPL} only by non-vanishing rescalings $\sqrt{g},(-2)$ which preserve ellipticity. Moreover, adding lower order derivatives to a differential operator won't affect its ellipticity either. So we can make the replacement with $\Psi$ in \eqref{BVPL} and also modify the last boundary term in \eqref{BVPB} to be the last one listed in \eqref{CW}. The resulting differential operator $P'=(L',B')$:
\begin{equation}\label{BVP1}
\begin{split}
&L':[S^{m,\alpha}_{\delta}\times T^{m,\alpha}_{\delta}\times C^{m,\alpha}_{\delta}](M)\to [S^{m-2,\alpha}_{\delta+2}\times C^{m-2,\alpha}_{\delta+2}\times (\wedge_1)^{m-2,\alpha}_{\delta+2}](M)\\
&\quad\quad L'(h,Y,v)=
\big(
\cE_0(h,v),~
\Psi(h,Y)
\big),\\
&B':[S^{m,\alpha}_{\delta}\times T^{m,\alpha}_{\delta}\times C^{m,\alpha}_{\delta}](M)\to [(\wedge_1)^{m-1,\alpha}\times S^{m,\alpha}\times (C^{m,\a})^2\times (\wedge_1)^{m-1,\alpha}](\p M)\\
&\quad\quad B'(h,Y,v)=\big(
\delta h-3d v,~
h^T,~
H'_{h},~
\tr^T[\delta^*Y+(\delta Y)g],~
\delta^*Y(\mathbf{n})^t+K(\mathbf n'_h)^T
\big)
\end{split}
\end{equation}
is also elliptic, which further implies that the map $\mathcal P$ defined below is Fredholm:
\begin{equation}\label{CP}
\begin{split}
&\mathcal P: \mathcal W\to S^{m-2,\alpha}_{\delta+2}(M)\times\mathcal T,\\
\mathcal P(h&,Y,v)=\big(
\cE_0(h,v),~
\Psi(h,Y)
\big).
\end{split}
\end{equation}
Thus the range $\text{Im}\mathcal P$ is closed and has finite codimension in the target space $S^{m-2,\alpha}_{\delta+2}(M)\times\mathcal T$. Let $\Pi$ be the projection 
$\Pi: S^{m-2,\alpha}_{\delta+2}(M)\times\mathcal T\to\mathcal T.$ Projecting the image of $\mathcal P$ to the second component, we obtain $\Pi(\text{Im}\mathcal P)=\text{Im}\Psi$. It must also be of finite codimension in $\mathcal T$. This completes the proof of the closed range of $\cL$.

\medskip

Now to prove the surjectivity of $\cL$ it suffices to show $\cL$ has trivial cokernel. We prove this by contradiction. 
Suppose ${\rm Coker}\cL$ is non-trivial. Then by the Hahn-Banach Theorem, there is a nontrivial element $\hat X$ in the dual space $\mathcal T^*$ so that
\begin{equation}\label{c11}
\begin{split}
\hat X[\cL(h,p)]=0,\mbox{ for all}~ (h,p)\in T\mathcal B.
\end{split}
\end{equation}
Here $\hat X$ can be decomposed as $\hat X=(X^0,X)~(X=X^i,~i=1,2,3)$ where $X^0\in \big(C^{m-2,\alpha}_{\delta+2}(M)\big)^*$ and $X\in\big((\wedge_1)^{m-2,\alpha}_{\delta+2}(M)\big)^*$ such that 
\begin{equation}\label{c12}
\begin{split}
\begin{cases}
X^0[(D\Phi_0)_{(g,K)}(h,p)]=0\\
X[(D\Phi_i)_{(g,K)}(h,p)]=0
\end{cases}\mbox{ for all } (h,p)\in T\mathcal B|_{(g,K)}.
\end{split}
\end{equation}
Here $(D\Phi_0)_{(g,K)}(h,p)),~(D\Phi_i)_{(g,K)}(h,p))$ are components of the linearized constraint map given in \eqref{lu}-\eqref{lz}; and we use $X^0[\cdot],X[\cdot]$ to denote the distributional pairing.
We first prove that $X^0,X$ are $C^{m,\alpha}$ smooth in the interior ${\rm int}M$ of $M$. So the pairings above are actually integral pairings.
Based on the construction of the space $\cW$ and the map $\Psi$, we observe that \eqref{c11} implies $\hat X[\Psi(h,Y)]=0$ for all $(h,Y,v)\in\cW$. It then follows trivially from the construction of $\cP$ that $(0,\hat X)$ is a cokernel element of $\mathcal P$, i.e. the pairing
$$(0,\hat X)[\mathcal P(h,Y,v)]=0,\mbox{ for all } (h,Y,v)\in \mathcal W.$$
Thus $(0,\hat X)$ is a weak solution of the elliptic equation
\begin{equation}\label{c3}
\begin{split}
\mathcal P^*(0,\hat X)=0
\end{split}
\end{equation}
in the interior of $M$. 
We follow the approach in \cite{LM} (Chapter 2, Theorem 3.2) to improve regularity of $\hat X$.
Take bounded open domains $V,U\subset {\rm int}M$ such that $V\subset \bar V\subset U$. Using the chart $M\cong\mathbb R^3\setminus B$, we can identify $U$ as a bounded domain in $\mathbb R^3$ and $X^{\mu}~(\mu=0,1,2,3)$ as distributions in $U$. Let $\phi$ be a smooth cutoff function which equals $1$ in $V$ and compactly supported in $U$. So $Y^{\mu}=\phi X^{\mu}$ are compactly supported distributions in $\mathbb R^3$ and  $(0,Y)$ is a weak solution to the elliptic equation \eqref{c3} inside $V$. Take the Fourier transform of $Y^{\mu}$
\begin{equation*}
\begin{split}
F(Y^{\mu})(y)=\frac{1}{(2\pi)^{3/2}}\int_{\mathbb R^3}e^{-ix\cdot y}Y^{\mu}(x)dx.
\end{split}
\end{equation*}
Let ${\w Y}^{\mu}$ be distributions in $\mathbb R^3$ such that their Fourier transforms are given by
\begin{equation*}
\begin{split}
F({\w Y}^{\mu})(y)=\big(\frac{1}{1+|y|^2}\big)^kF(Y^{\mu})(y)
\end{split}
\end{equation*}
It then follows that
\begin{equation}\label{cz1}
(I+\Delta_0)^k{\w Y}^{\mu}=Y^{\mu}\quad (\mu=0,1,2,3)\ \ {\rm in}~\bR^3.
\end{equation}
Here $\Delta_0$ is the Laplacian with respect to the flat metric in $\mathbb R^3$. The order $k$ is chosen so that $2k\geq m+1$. 
Based on Sobolev embedding $H^{m+1}(\mathbb R^3)\subset C^{m-2,\alpha}(\mathbb R^3)$ for $m\geq 2,\alpha\in(0,1)$, we have ${Y}^{\mu}\in(C^{m-2,\alpha}(\mathbb R^3))^*\subset H^{-m-1}(\mathbb R^3)$, and hence by construction (choice of $k$) ${\w Y}^{\mu}\in L^2$. 

Now ${\w Y}^{\mu}$ are $L^2$ functions solving the following elliptic system in $V$ 
\begin{equation}\label{cz2}
\begin{split}
\mathcal P^*(0,(I+\Delta_0)^k {\w Y}^{\mu})=0.
\end{split}
\end{equation}
Note that the coefficients in the elliptic system above are at least $C^{m-2,\alpha}$ smooth . Thus by interior regularity for elliptic equations (cf.\cite{Morrey} Theorem 6.2.6), we can obtain ${\w Y}^{\mu}\in C^{m+2k,\alpha}(V')$ for any compact subset $V'\subset V$. Then it follows from equation \eqref{cz1} that $Y^{\mu}\in C^{m,\alpha}(V')$. By a partition of unity argument, it is easy to see that $X^{\mu}$ is $C^{m,\alpha}$ smooth in ${\rm int}M$.

Next we prove $\hat X=0$ in ${\rm int}M$. 
By basic computation with integration by parts, \eqref{c12} implies that $\hat X$ is a solution of the following equations on ${\rm int}M$ (cf. for example \cite{Mo}):
\begin{equation}\label{c2}
\begin{cases}
2X^0K+L_{X}g=0,\\
D^2X^0+L_{X}K+X^0[-Ric_g+2K\circ K-(trK)K+\tfrac{1}{4}ug]=0.
\end{cases}
\end{equation}
Since $\hat X$ is $C^{m,\a}$ in ${\rm int}M$ and by assumption $m>2$, according to Proposition 2.1 in \cite{BC}, there exist constants $\Lambda_{\mu\nu}=\Lambda_{[\mu\nu]}~(\mu,\nu=0,1,2,3)$ such that
\begin{equation}\label{ax1}
\begin{split}
X^i-\Lambda_{ij}x^j\in C^m_{\d-1}(M),\quad X^0-\Lambda_{0i}x^i\in C^m_{\d-1}(M);
\end{split}
\end{equation}
or there exist constants $A^{\mu}$ such that 
\begin{equation}\label{ax2}
\begin{split}
X^i-A^i\in C^m_{\d}(M),\quad X^0-A^0\in C^m_{\d}(M).
\end{split}
\end{equation}
On the other hand, $\hat X$ is also a bounded linear functional on $\cT$, so we must have $\Lambda_{\mu\nu}=A^{\mu}=0$ (cf. appendix \S4.5 for the detailed proof). Then according to the same proposition of \cite{BC}, we must have $\hat X=0$ in ${\rm int}M$.
Therefore $ \hat X[(\bar u,\bar Z)]=0$ for any compactly supported $(\bar u,\bar Z)\in\mathcal T$ and hence the same for any $(u_0,Z_0)\in\mathcal T$ which vanishes on $\partial M$.

Furthermore, it is easy to show that any $(u,Z)\in\mathcal T$ can be decomposed as $(u,Z)=(u_0,Z_0)+(u_1,Z_1)$ where $(u_0,Z_0)$ vanishes on the boundary and $(u_1,Z_1)\in\text{Im}\cL$ (cf. appendix \S4.2 for a detailed proof). 
So $\hat X[(u,Z)]= \hat X[(u_0,Z_0)]+ \hat X[(u_1,Z_1)]=0$ for all $(u,Z)\in\mathcal T$, i.e. $\hat X=0$. This completes the proof of trivial cokernel.

Summing up the results above, we have proved that 
\begin{Proposition}
The linearized constraint map given in \eqref{lconsm} is surjective.
\end{Proposition}
\subsection{Splitting Kernel} We apply the approach developed in \cite{W} to prove the kernel of the linearized constraint map $\cL=D\Phi|_{(g,K)}$ at $(g,K)\in \Phi^{-1}(u,Z)$ splits in the domain space $T\cB|_{(g,K)}$. 
We first state the following proposition.
\begin{proposition}
The tangent space $T\mathcal B|_{(g,K)}$ given in \eqref{lBcon} admits a splitting
\begin{equation}\label{DT1}
\begin{split}
T\mathcal B|_{(g,K)}=S_1\oplus S_2
\end{split}
\end{equation}
where $S_1,S_2$ are closed subspaces such that the range of the restricted map $\cL|_{S_1}:S_1\to\cT$ has finite codimension. Moreover, the kernel of $\cL$ splits in $S_1$, i.e. there is a closed subspace $S$ such that
\begin{equation}\label{KS1}
\begin{split}
S_1=S\oplus [\cL^{-1}(0)\cap S_1].
\end{split}
\end{equation}
\end{proposition}
Assuming the above proposition holds, we can then prove:
\begin{proposition}
The linearized constraint map given in \eqref{lconsm} has splitting kernel.
\end{proposition}
\begin{proof}
Here we apply the approach developed in \cite{W}. 
Decompose the target space $\mathcal T$ as 
\begin{equation}\label{DT}
\mathcal T=\cL(S_1)\oplus K
\end{equation} 
with $\text{dim} K<\infty$. Based on the decomposition \eqref{KS1}, the restricted linear map $\cL|_S$ given by
\begin{equation*}
\begin{split}
\cL|_{S}:S\to \cL(S_1)
\end{split}
\end{equation*} 
is bounded and bijective. By the open map theorem, it admits a bounded inverse denoted by $\w \cL=(\cL|_S)^{-1}$.
Let $\w\pi$ denote the projection from $S_1$ onto $[\cL^{-1}(0)\cap S_1]$, and $\pi_K$ denote the projection from $\mathcal T$ onto $K$. We then obtain the following description of the kernel ${\rm Ker}\cL$ in $T\mathcal B|_{(g,K)}$
\begin{equation*}
\begin{split}
&{\rm Ker}\cL=\{(h,p)\in T\mathcal B|_{(g,K)}:~\cL (h,p)=0\}\\
&=\{(h,p)=(h_1,p_1)+(h_2,p_2):~(h_1,p_1)\in S_1,~(h_2,p_2)\in S_2,\text{ and }\cL (h_1,p_1)=-\cL(h_2,p_2)\}\\
&=\{(h,p)=(h_1,p_1)+(h_2,p_2):~(h_1,p_1)\in S_1,~(h_2,p_2)\in{\rm Ker }(\pi_K\circ \cL)\cap S_2,~\cL(h_1,p_1)=-\cL(h_2,p_2)\}\\
&=\{(h,p)=(h_1,p_1)+(h_2,p_2):(h_2,p_2)\in\text{Ker}(\pi_K\circ \cL)\cap S_2,~(h_1,p_1)=\w \cL\big(-\cL(h_2,p_2)\big)+\w\pi(h_1,p_1)\}.
\end{split}
\end{equation*}
The third equality above is based on that $\cL (h_1,p_1)=-\cL(h_2,p_2)$ implies $\cL(h_2,p_2)\in \cL(S_1)$ and hence $\pi_K\circ \cL(h_2,p_2)=0$ according to \eqref{DT}. 
In the last equality we use the inverse map $\w \cL$ and projection $\w\pi$ to express $(h_1,p_1)\in S_1$ based on $\cL(h_1,p_1)=-\cL(h_2,p_2)$.
Since the map $\pi_K\circ \cL: S_2\to K$ has a target space of finite dimension, its kernel must be of finite codimension and hence splits in $S_2$. So there is a bounded projection $P$ from $S_2$ onto $\text{Ker}(\pi_K\circ \cL)\cap S_2$. Then we obtain a bounded projection from $T\mathcal B|_{(g,K)}$ onto $\text{Ker}\cL$ given by
\begin{equation*}
\begin{split}
&\Pi:T\mathcal B|_{(g,K)}\to {\rm Ker}\cL\\
\Pi:(h,p)=(h_1,p_1)+(h_2,&p_2)\mapsto\{P(h_2,p_2)+\w \cL[-\cL\big(P(h_2,p_2)\big)]+\w\pi(h_1,p_1)\}.
\end{split}
\end{equation*}
This completes the proof of splitting kernel.
\end{proof}

\medskip

Now we give the proof of Proposition 2.3. First notice that the space of 1-forms on the boundary manifold $\p M$ can be decomposed as $\wedge_1(\partial M)=\text{Im}d^T\oplus\text{Ker}\delta^T$. Here $d^T$ denotes the exterior derivative operator on the boundary $d^T: C^{m,\a}(\p M)\to (\wedge_1)^{m-1,\a}(\p M)$; and $\d^T$ denotes the divergence operator $\d^T:(\wedge_1)^{m-1,\a}(\p M)\to C^{m-2,\a}(\p M)$ with respect to the induced metric $g^T$. So for the 1-form $(\delta h)^T$ on $\p M$ induced by a general symmetric 2-tensor $h$ on $M$, there is $v_h\in C^{m,\alpha}(\partial M)$ and $\tau_h\in\text{Ker}\delta^T$ on $\partial M$ such that 
\begin{equation}\label{DH1}
(\delta h)^T=d^Tv_{h}+\tau_{h},
\end{equation}
where the 1-forms $d^Tv_h$ and $\tau_h$ are uniquely determined by $h$. Construct a bounded linear map
\begin{equation*}
\begin{split}
E_1: \text{Ker}\delta^T\to S^{m,\alpha}_{\delta}(M),
\end{split}
\end{equation*}
so that for any $\tau\in\text{Ker}\delta^T$, $\bar h=E_1(\tau)$ is a symmetric 2-tensor on $M$ and the following conditions hold
\begin{equation}\label{LE1}
(\delta \bar h)^T=\tau,\ \
\bar h^T=0,\ \
H'_{\bar h}=0,\ \
\mathbf n'_{\bar h}=0\quad\text{ on }\partial M.
\end{equation}
{There are many ways to construct such a map.  We refer to \S 4.3 for an appropriate candidate.}
Now given an element $(h,p)\in T\mathcal B|_{(g,K)}$, we can decompose $h$ as  
\begin{equation}\label{DH2}
h=[h-E_1(\tau_h)]+E_1(\tau_h)
\end{equation} 
where $\tau_h$ is uniquely determined by $h$ as in \eqref{DH1}. Notice that for the first part in the summation above, we have 
$\big(\delta[h-E_1(\tau_h)]\big)^T
=(\d h)^T-\tau_h
=d^Tv_h$ 
on $\partial M$. So $\big(\delta[h-E_1(\tau_h)]\big)^T\in{\rm Im}d^T$. Then it is easy to construct a scalar field $v$ on $M$ such that $\delta[h-E_1(\tau)]=3dv$ along the boundary, so that $\big([h-E_1(\tau_h)],v\big)$ satisfies the gauge condition in $\cW$ (cf. \eqref{CW}).

Next construct a bounded linear (0-order in $h$) map
$E_2:S^{m,\alpha}_{\delta}(M)\to S^{m,\alpha}_{\delta+1}(M)$
such that for any $h\in S^{m,\alpha}_{\delta}(M)$, $\w h=E_2(h)$ is a symmetric 2-tensor belonging to $S^{m,\a}_{\d+1}(M)$ and satisfying the following boundary conditions
\begin{equation}\label{LE2}
\tr^T\w h=0,\ \
\w h(\mathbf n)^T=-K(\mathbf n'_h)^T
\quad\text{on }\partial M.
\end{equation}
Just as for the map $E_1$, we refer to \S 4.3 for a possible construction of $E_2$. Now given an element $(h,p)\in T\mathcal B|_{(g,K)}$, one can first decompose $h$ as in equation \eqref{DH2} and then decompose $p$ as
\begin{equation}\label{DP1}
p=E_2[h-E_1(\tau_h)]+\big(p-E_2[h-E_1(\tau_h)]\big).
\end{equation}
Note $E_2$ is constructed such that its image is of decay rate $\delta+1$, which is the same as $p$. 
Moreover, combining conditions \eqref{LE1} and \eqref{LE2} we can derive that in the decomposition above the second component $\big(p-E_2[h-E_1(\tau_h)]\big)$ belongs to the subspace
\begin{equation}\label{S0}
\begin{split}
S_0=\{p\in S^{m-1,\alpha}_{\delta+1}(M): \tr^Tp=0,~p(\mathbf n)^T=0\text{ on }\partial M\}.
\end{split}
\end{equation}
We have the following lemma for this space.
\begin{lemma} The space $S_0$ defined in \eqref{S0} admits the following splitting:
\begin{equation*}
\begin{split}
S_0={\rm Im} Q\oplus {\rm Ker} Q^*
\end{split}
\end{equation*}
where $Q$ is the linear differential operator given by
\begin{equation*}
\begin{split}
&Q:T_0\to S_0\\
Q(&Y)=\delta^*Y+(\delta Y)g,
\end{split}
\end{equation*}
with $T_0=\{Y\in T^{m,\alpha}_{\delta}(M):\tr^T(\delta^*Y+(\delta Y)g)=0,\delta^*Y(\mathbf n)^T=0\text{ on }\partial M\}$. 
\end{lemma}
\begin{proof}
The formal adjoint of $Q$ is $Q^*=\delta+d\tr$, acting on the space of symmetric 2-tensors $p\in S_0$. For $Y\in T_0$ and $p\in S_0$ the following equality holds
\begin{equation*}
\begin{split}
\int_{M}\langle\delta^*Y+(\delta Y)g,p\rangle d\vol_g
&=\int_{M}\langle Y, \delta p+dtrp \rangle d\vol_g+\int_{\partial M}p(\mathbf n,Y) -Y(\mathbf n)trp~d\vol_{g^T}\\
&=\int_{M}\langle Y, \delta p+dtrp \rangle d\vol_g,
\end{split}
\end{equation*}
where the boundary integral vanishes because
$p(\mathbf n,Y) -Y(\mathbf n)\tr p=p(\mathbf n,\mathbf n) Y(\mathbf n) -Y(\mathbf n)\tr p=-Y(\mathbf n)\tr^Tp=0$.
Besides, since $Y,p$ decay fast enough asymptotically, we have vanishing boundary term at infinity from the integration by parts.
It follows that $Q^*Q$ is a self-adjoint elliptic operator. In addition $\text{Ker}Q^*Q=\text{Ker}Q$ since $\int_M \<Q^*QY,Y\> d\vol_g=\int_M\<QY,QY\>d\vol_g$. Thus for any $p_0\in S_0$
\begin{equation*}
\begin{split}
\int_{M}\langle Q^*(p_0), Y\rangle d\vol_g=0\mbox{ for all } Y\in\text{Ker}Q^*Q,
\end{split}
\end{equation*}
i.e. $Q^*(p_0)$ is perpendicular to the kernel of $Q^*Q$. By self-adjointness of $Q^*Q$, ${\rm Ker}(Q^*Q)={\rm Coker}(Q^*Q)$. Thus $Q^*(p_0)\in{\rm Im}(Q^*Q)$, i.e. there exists a vector field $Y_0\in T_0$ such that 
$Q^*p_0=Q^*QY_0$.
So $p_0=QY_0+w_0$ with $w_0\in \text{Ker} Q^*$.
Furthermore, it is easy to check this decomposition is unique because ${\rm Im}Q\cap{\rm Ker}Q^*=\{0\}$.
\end{proof}

Going back to the decomposition of $p$ \eqref{DP1}, we can apply the lemma above to the second component . So $p$ can be further decomposed as
\begin{equation}\label{DP2}
p=E_2[h-E_1(\tau_h)]+Q(Y_p)+w_p.
\end{equation}
Here $Y_p\in T_0$ and $w_p\in{\rm Ker}Q^*$, and both of them are uniquely determined by $(h,p)$. Summing up the analysis above, we conclude that every element $(h,p)\in T\mathcal{B}|_{(g,K)}$ admits the following decomposition:
\begin{equation*}
\begin{split}
(h,p)=\big(~h-E_1(\tau_h),~E_2[h-E_1(\tau_h)]+Q(Y_p)~\big)+\big(~E_1(\tau_h),~w_p~\big).
\end{split}
\end{equation*}
It follows that 
\begin{equation}\label{DT2}
\begin{split}
T\mathcal{B}|_{(g,K)}=S_1+ S_2,
\end{split}
\end{equation}
where
\begin{equation*}
\begin{split}
&S_1=\{(h,p)\in T\mathcal{B}|_{(g,K)}:~ (\d h)^T\in{\rm Im}d^T \text{ on }\p M,~p=E_2(h)+Q(Y)\text{ on }M\text{ for some }Y\in  T_0~\},\\
&S_2=\{(h,p)\in T\mathcal{B}|_{(g,K)}:~h\in\text{Im}E_1,~p\in{\rm Ker}Q^*\}.
\end{split}
\end{equation*}
Then equation \eqref{DT1} in Proposition 2.3 will be true if the following lemma holds.
\begin{lemma} Equation \eqref{DT2} is a splitting of $T\cB|_{(g,K)}$, i.e. $S_1, S_2$ are closed subspaces and their intersection is trivial.
\end{lemma}
\begin{proof} 
Observe $S_1,S_2$ are well-defined subspaces of $T\mathcal B|_{(g,K)}$. It suffices to show the following:
\begin{enumerate}
\item The intersection $S_1\cap S_2=\{0\}$. Assume $(h_0,p_0)\in S_1\cap S_2$. 
So $(\delta h_0)^T=d^Tv_0$ on $\p M$ for some scalar field $v_0$. On the other hand, there exist $\tau_0\in\text{Ker}\delta^T$ such that $h_0=E_1(\tau_0)$. It follows that 
$\tau_0=(\delta h_0)^T=d^Tv_0$ on $\partial M$. Then $\delta^Td^Tv_0=0$, which implies that $v_0$ is a constant function and hence $\tau_0=0$. Thus $h_0=E_1(0)=0$ and it follows that $p_0\in{\rm Im}Q\cap{\rm Ker}Q^*$ which further implies that $p_0=0$.

\item The subspace $S_1$ is closed. Suppose there is a sequence $\{(h_i,p_i)\}~(i=1,2,3,...)$ in $S_1$ that converges to $(h_0,p_0)\in T\mathcal B|_{(g,K)}$. For every $i$, $(\delta h_i)^T\in {\rm Im}d^T$ on the boundary. So $(\delta h_i)^T$ is a closed 1-form on $\partial M$. It follows that $(\delta h_0)^T$ is also closed and hence exact i.e. $(\delta h_0)^T\in{\rm Im}d^T$. Secondly, for each $i$ there exists $Y_i\in T_0$ such that $p_i=E_2(h_i)+Q(Y_i)$; and convergence of $h_i$ and $p_i$ implies the sequence $Q(Y_i)=p_i-E_2(h_i)$ converges to $p_0-E_2(h_0)$. Since the range of $Q$ is closed, there exists some $Y_0\in  T_0$ such that $p_0-E_2(h_0)=Q(Y_0)$. So we can conclude the limit $(h_0,p_0)\in S_1$.

\item The subspace $S_2$ is closed. Obviously ${\rm Ker}Q^*$ is closed. In addition the map $E_1$ must also have closed range, because for any 1-form $\tau\in{\rm Ker}\d^T$ on $\p M$ we have $\tau=\big(\d[E_1(\tau)]\big)^T$, i.e. the norm of $\tau$ is controlled by the norm of its image $E_1(\tau)$.
 This completes the proof.
\end{enumerate}
\end{proof}

Next we prove the properties of the subspace $S_1$ stated in Proposition 2.3. Define the following space 
\begin{equation}\label{CW1}
\begin{split}
\mathcal W'=\{&(h,Y, v)\in [S^{m,\alpha}_{\delta}\times  T^{m,\alpha}_{\delta}\times C^{m,\alpha}_{\delta}](M):\\
 &~\delta h-3dv=0,~h^T=0,~H'_h=0,~\tr^T(\d^* Y)+2(\d Y)=0,~\d^*Y({\bf n})^T=0
\text{ on }\partial M
\}.
\end{split}
\end{equation}
Notice that the only difference between $\cW'$ and  $\mathcal W$ in \eqref{CW} is lower order terms of $h$ in the last boundary equation. As previously, let $\mathcal V'$ denote the space of pairs $(h,Y)$ such that $(h,Y,v)\in\mathcal W'$ for some function $v$. Then it is easy to observe that 
\be\label{CV2}
(\d h)^T\in{\rm Im}d^T,~Y\in T_0\ \ \mbox{ for all }(h,Y)\in\cV',
\ee
and the subspace $S_1$ in \eqref{DT2} can be equivalently written as 
\be\label{S1}
S_1=\{(h,p)\in T\cB|_{(g,K)}:~h=h,~ p=E_2(h)+Q(Y)\text{ for some }(h,Y)\in\cV'\}.
\ee
Via this formula, we can construct a new operator
\begin{equation}
\begin{split}
\hat\Psi:&\mathcal V'\to\mathcal T\\
\hat\Psi(h,Y)=\cL\big(h,E_2(h)+Q(Y)\big)
&=\cL\big(h,\delta^*Y+(\delta Y)g\big)+\cL(0,E_2(h))
=\Psi(h,Y)+O_1(h),
\end{split}
\end{equation}
where the formula for $\Psi(h,Y)$ is the same as in equation \eqref{PS1}-\eqref{PS2}, and $O_1(h)$ denotes a term only involving zero and first order derivatives of $h$. 
Define an ``Einstein-type'' operator $\mathcal E$ on the space $\mathcal W'$ similar to the operator $\cP$ in \eqref{CP}:
\begin{equation}\label{CE}
\begin{split}
&\cE:\cW'\to S^{m-2,\a}_{\d+2}(M)\times\cT\\
\cE(&h,Y,v)=(~\cE_0(h,v),~\hat\Psi(h,Y)~).
\end{split}
\end{equation}
Notice that the leading order part of $\cE$ is the same as that of $L$ in \eqref{BVPL} and the domain space $\cW'$ consists of exactly kernel elements of the operator $B$ in \eqref{BVPB} by construction. It follows from the ellipticity of $P=(L,B)$ that $\mathcal E$ is a Fredholm map and hence its range has finite codimension. Let $\pi_2$ be the projection to the second component in \eqref{CE}. Obviously the image of $\pi_2\circ\mathcal E$ is equal to the range $\cL(S_1)$. Therefore, $\cL(S_1)$ also has finite codimension in $\cT$, as stated in Proposition 2.3.

Lastly using the map $\mathcal E$ defined above we give the proof of equation \eqref{KS1}.
\begin{lemma}The subspace $S_1\cap \cL^{-1}(0)$ splits in $S_1$, i.e.  
\begin{equation*}
\begin{split}
S_1= S\oplus \big(S_1\cap \cL^{-1}(0)\big)
\end{split}
\end{equation*}
for some closed subspace $S\subset S_1$.
\end{lemma}
\begin{proof} The following proof is based on the equivalent description \eqref{S1} for the space $S_1$.
Notice that \eqref{S1} relates $S_1$ with the space $\cV'$ and $\cW'$.
So we will first construct a splitting for $\cW'$ and $\cV'$, and then derive the splitting for $S_1$.
Let $W_1$ be the subspace of $\mathcal W'$ which consists of elements $(h,Y,v)$ such that $\hat\Psi(h,Y)=0$, i.e. $ W_1=\mathcal E^{-1}(*,0)$.
Similarly, define $W_2=\mathcal E^{-1}(0,*)$ as the space consisting of $(h,Y,v)$ such that $\cE_0(h,v)=0$. Then $W_1$ and $W_2$ are closed subspaces of $\mathcal W'$. Moreover, $(W_1+W_2)$ must be of finite codimension in $\mathcal W'$. In fact, we can construct a map 
\begin{equation*}
\begin{split}
\cF:\cW'/(&W_1+W_2)\to \big(S^{m-2,\a}_{\d+2}(M)\times\cT\big)/{\rm Im}\cE\\
&\cF([h,Y,v])=[\cE_0(h,v),0]
\end{split}
\end{equation*}
where $[h,Y,v]$ denotes an equivalence class in $\cW'/(W_1+W_2)$ and $[\cE_0(h,v),0]$ an equivalence class in $\big(S^{m-2,\a}_{\d+2}(M)\times\cT\big)/{\rm Im}\cE$. It is easy to verify $\cF$ is well-defined and injective. 
Since the range of the Fredholm map $\mathcal E$ has finite codimension, the quotient $\big(S^{m-2,\a}_{\d+2}(M)\times\cT\big)/{\rm Im}\cE$ and hence $\cW'/(W_1+W_2)$ must be of finite dimension. Let $W_3$ be a complementary subspace, i.e. $$\mathcal W'=(W_1+W_2)\oplus W_3.$$
Notice that $W_1\cap W_2=\mathcal E^{-1}(0,0)$ is of finite dimension and thus it splits in $W_2$,
i.e. $W_2=(W_1\cap W_2)\oplus \w W_2$ for some closed subspace $\w W_2$. 
This further implies that 
\begin{equation}\label{DW}
\mathcal W'=W_1\oplus \w W_2\oplus W_3.
\end{equation}
Now consider the previously defined space $\mathcal V'$.
We will show that $V_1=\hat\Psi^{-1}(0)$ splits in $\mathcal V'$. Let $\pi$ be the projection
$\pi:\mathcal W'\to\mathcal V',\ 
\pi(h,Y,v)=(h,Y)$.
Obviously, $V_1=\pi(W_1)$ and 
\begin{equation}\label{DV1}
\mathcal V'=V_1+\pi(\w W_2)+\pi(W_3).
\end{equation}
Let $V_2=\pi(\w W_2)$. It follows from the definition of $\w W_2$ that $V_1\cap V_2=\{0\}$. Moreover, $V_2$ is also closed. In fact, given a Cauchy sequence $\{(h_i,Y_i)\}~(i=1,2,3...)$ in $V_2$, there is a sequence $\{v_i\}$ such that $(h_i,Y_i,v_i)\in\w W_2$. The sequence of their images $\mathcal E(h_i,Y_i,v_i)=(0,\hat\Psi(h_i,Y_i))$ must also converge since $\hat\Psi$ is a bounded operator. 
Observe $\mathcal E|_{\w W_2}:~\w W_2\to \{(0,*)\}\cap\text{Im}\mathcal E$ is a bijective and bounded linear operator. It then follows that $(h_i,Y_i,v_i)$ must converge to some element $( h_0,  Y_0, v_0)$ in $\w W_2$; and consequently $(h_i,Y_i)$ converges to $( h_0, Y_0)$ in $V_2$.
Thus equation \eqref{DV1} can be rewritten as
$\mathcal V'=(V_1\oplus V_2)+\pi(W_3)$.
In this decomposition $V_1\oplus V_2$ must be of finite codimension, since $W_3$ has finite dimension. Thus there exist a closed subspace $V_3$ 
so that 
\begin{equation}\label{DV3}
\mathcal V'=V_1\oplus V_3.
\end{equation}
Finally, combining the splitting \eqref{DV3} for $\cV'$ and description \eqref{S1} for $S_1$, we can finish the proof of the lemma. Define the map 
\begin{equation*}
\begin{split}
&T:\mathcal V'\to S_1\\
T(h,Y)&=(h,E_2(h)+Q(Y)).
\end{split}
\end{equation*}
Obviously $T$ is linear, bounded and surjective with kernel given by $\text{Ker} T=\{(0,Y)\in\mathcal{V'}:Q(Y)=0\}$. Since $\hat\Psi(h,Y)=\cL(T(h,Y))$, we have $T(V_1)=\cL^{-1}(0)\cap S_1.$ Thus
$$S_1=(\cL^{-1}(0)\cap S_1)+T(V_3).$$
According to \eqref{DV3} we see that $(\cL^{-1}(0)\cap S_1)\cap T(V_3)=\{0\}$. So \eqref{KS1} will hold if $T(V_3)$ is closed. Suppose $\{\big(h_i,p_i\big)\}~(i=1,2,3...)$ is a Cauchy sequence in $T(V_3)$, so $p_i= E_2(h_i)+Q(Y_i)$. Then $\hat\Psi(h_i,Y_i)=\cL\big(h_i,E_2(h_i)+Q(Y_i)\big)$ must converge in $\cL(S_1)$ since $\cL$ is bounded. 
Then $(h_i,Y_i)$ must converge to some $(h_0,Y_0)$ in $V_3$ because the map $\hat\Psi|_{V_3}:V_3\to \text{Im}\hat\Psi=\cL(S_1)$ is bounded, linear and bijective. Therefore $\big(h_i, p_i\big)$ must converge to $\big( h_0, E( h_0)+Q( Y_0)\big)$ in $T(V_3)$. This completes the proof.
\end{proof}

\medskip

Summarizing all the previous results, we can conclude that level sets of the constraint map admit Banach manifold structure.
\begin{theorem}
Given fixed Bartnik data $(\gamma,l,k,\tau)\in [Met^{m,\a}\times C^{m-1,\a}\times C^{m-1,\a}\times (\wedge_1)^{m-1,\a}](\p M)$ on $\p M$ and $(u,Z)\in \cT$, the space $\mathcal C_B(u,Z)$ of initial data sets satisfying the constraint equations with fixed boundary data
\begin{equation*}
\begin{split}
\mathcal C_B(u,Z)=\{(g,K)\in[Met_{\delta}^{m,\alpha}\times S^{m-1,\alpha}_{\delta+1}](M):&~\Phi(g,K)=(u,Z)\text{ on }M\\
&(g^T,H,tr_{\partial M}K,K(\mathbf n)^T)=(\gamma,l,k,\tau)\text{ on }\partial M.\}
\end{split}
\end{equation*}
is an infinite dimensional smooth Banach manifold. 
\end{theorem}
\begin{proof}
According to Proposition 2.2 and Proposition 2.4, the linearization $\cL=D\Phi|_{(g,K)}$ at any $(g,K)\in\Phi^{-1}(u,Z)$ is surjective and has splitting kernel. The theorem is a natural consequence of the implicit function theorem in Banach spaces.
\end{proof}
\subsection{Ellipticity of the ``Einstein-type" operator} In the last part of this section, we prove in detail that the operator $P$ constructed as \eqref{BVPL}-\eqref{BVPB} in the proof of surjectivity is elliptic. First observe that in \eqref{BVPL}-\eqref{BVPB} the vector field $Y$ is not coupled with $(h,v)$. So we can split $P$ as an operator $P_1=(L_1,B_1)$ acting on $Y$ given by
\begin{equation}\label{EY}
\begin{split}
&L_1:  T^{m,\alpha}_{\delta}(M)\to (\wedge_1)^{m-2,\alpha}_{\delta+2}(M)\\
&\quad\quad L_1(Y)=\delta\delta^*Y+d\delta Y\\
&B_1:  T^{m,\alpha}_{\delta}(M)\to [C^{m-1,\a}\times (\wedge_1)^{m-1,\alpha}](\p M)\\
&\quad\quad B_1(Y)=\big(
\tr^T[\delta^*Y+(\delta Y)g],~
\delta^*Y(\mathbf{n})^T
\big),
\end{split}
\end{equation}
and an operator $P_2=(L_2,B_2)$ acting on $(h,v)$ given by
\begin{equation}\label{EH}
\begin{split}
&\quad\quad L_2:[S^{m,\alpha}_{\delta}\times C^{m,\alpha}_{\delta}](M)\to [S^{m-2,\alpha}_{\delta+2}\times C^{m-2,\alpha}_{\delta+2}](M)\\
& \quad\quad\quad\quad L_2(h,v)=
\big(
\cE_0(h,v),~
\Delta \tr h+\delta\delta h
\big),\\
&B_2:[S^{m,\alpha}_{\delta}\times C^{m,\alpha}_{\delta}](M)\to [(\wedge_1)^{m-1,\alpha}\times S^{m,\alpha}\times C^{m,\a}](\p M)\\
&\quad\quad\quad\quad B_2(h,v)=\big(
\delta h-3d v,~
h^T,~
H'_{h}
\big).
\end{split}
\end{equation}
It is easy to verify the ellipticity of $P_1$ by applying the criterion given in \cite{ADN}. Here we give the detail.
According to \cite{ADN}, a general boundary value operator $P=(L,B)$ is elliptic
if the following two conditions hold:
\\(A) (properly elliptic condition): Let $L(\xi)$ denote the matrix of principal symbol of the interior operator $L$. Then its determinant $\ell(\xi)=\det L(\xi)$ has no nonzero real root at any point $x\in M$;
\\(B) (complementing boundary {condition}): Let $B(\xi)$ be the matrix of principal symbol of the boundary operator $B$ and $L^*(\xi)$ be the adjoint matrix of $L(\xi)$.
At any point $x\in\p M$, take $\xi=z\mu+\eta$ where $\eta$ denotes a non-zero 1-form tangential to the boundary $\p M$ and $\mu$ a unit 1-form normal to $\p M$. Then the rows of the product $B(z\mu+\eta)\cdot L^{\ast}(z\mu+\eta)$ are linearly independent modulo $\ell^+(z)$, where $\ell^+(z)=\prod_k(z-z_k)$ with $\{z_k\}$ being the roots of $\ell(z\mu+\eta)=0$ with positive imaginary parts.

At any point $x\in M$ we can choose the normal coordinates and express
\begin{align*}
&2(\d\d^*Y+d\d Y)= -\p_i\p_iY_j-\p_i\p_jY_i-2\p_j\p_iY_i+O_1(Y)=-\p_i\p_iY_j-3\p_i\p_jY_i+O_1(Y).
\end{align*}
Recall that $O_1(Y)$ denotes a term involving at most 1st order derivatives of $Y$. So the matrix of principal symbol of $L_1$ is given by
\begin{equation*}
L_1(\xi)=\frac{1}{2}
\begin{bmatrix}
|\xi|^2+3\xi_1\xi_1&3\xi_1\xi_2&3\xi_1\xi_3\\
3\xi_2\xi_1&|\xi|^2+3\xi_2\xi_2&3\xi_2\xi_3\\
3\xi_1\xi_3&3\xi_2\xi_3&|\xi|^2+3\xi_3\xi_3
\end{bmatrix},
\end{equation*}
where $|\xi|=\sqrt{\xi_1^2+\xi_2^2+\xi_3^2}$. Elementary calculation shows its determinant is 
$\ell_1(\xi)=\tfrac{1}{2}|\xi|^6$. Obviously it satisfies properly elliptic condition. 
The adjoint matrix is given by
\begin{equation*}
L^*_1(\xi)=\frac{1}{4}|\xi|^2
\begin{bmatrix}
|\xi|^2+3(\xi_2^2+\xi_3^2)&-3\xi_1\xi_2&-3\xi_1\xi_3\\
-3\xi_2\xi_1&|\xi|^2+3(\xi_1^2+\xi_3^2)&-3\xi_2\xi_3\\
-3\xi_1\xi_3&-3\xi_2\xi_3&|\xi|^2+3(\xi_1^2+\xi_2^2)
\end{bmatrix}.
\end{equation*}
At any point $x\in\p M$, again in the normal coordinates
\begin{align*}
&\tr^T[\d^*Y+(\d Y)g]=-2\p_1Y_1-\p_2Y_2-\p_3Y_3+O_0(Y),\\
&2\delta^*Y(\mathbf{n})^T_A=\p_1Y_A+\p_AY_1+O_0,\ \ A=2,3.
\end{align*}
Recall that we use the index $1$ to denote the normal direction to the boundary $\p M$ and indices $2,3$ to denote the directions tangential to $\p M$. 
Thus the matrix of principal symbol of $B_2$ is given by
\begin{equation*}
B_2(\xi)=i
\begin{bmatrix}
-2\xi_1&-\xi_2&-\xi_3\\
\xi_2&\xi_1&0\\
\xi_3&0&\xi_1
\end{bmatrix}.
\end{equation*}
Thus we have 
\begin{equation*}
\begin{split}
B_1(\xi)L^*_1(\xi)=\frac{i}{4}|\xi|^2
\begin{bmatrix}
-2|\xi|^2\xi_1-3\xi_1(\xi_2^2+\xi_3^2)&-|\xi|^2\xi_2+3\xi_1^2\xi_2&-|\xi|^2\xi_3+3\xi_1^2\xi_3\\
2\xi_2(2\xi_2^2+2\xi_3^2-\xi_1^2)&2\xi_1(-\xi_2^2+2\xi_1^2+2\xi_3^2)&-6\xi_1\xi_2\xi_3\\
2\xi_3(2\xi_2^2+2\xi_3^2-\xi_1^2)&-6\xi_1\xi_2\xi_3&2\xi_1(-\xi_3^2+2\xi_1^2+2\xi_2^2)
\end{bmatrix}.
\end{split}
\end{equation*}
Since $\ell_1(z\mu+\eta)=(z^2+|\eta|^2)^3$, the root with positive imaginary part for $\ell_1(z\mu+\eta)=0$ is $z=i|\eta|$ of multiplicity 3.
So the complementing boundary condition will be true if there is no nonzero complex vector $C$ solving $C\cdot B_1(z\mu+\eta)\cdot L_1^*(z\mu+\eta)=0\mbox{ mod }(z-i|\eta|)^3$.
Denote the matrix on the right side of the expression above as $\hat B(\xi)=\tfrac{4}{i|\xi|^2}B_1(\xi)L_1^*(\xi)$. Then it suffices to show there is no nontrivial solution for 
$C\cdot \hat B(z\mu+\eta)=0\mbox{ mod }(z-i|\eta|)^2$.

It is easy to verify that $\det \hat B(z\mu+\eta)=0\mbox{ mod } (z-i|\eta|)$, which means that the rows of $ \hat B(z\mu+\eta)$ are linearly dependent mod $(z-i|\eta|)$. Thus we need to take the derivative $\hat B'(z\mu+\eta)$ of $\hat B(z\mu+\eta)$ with respect to $z$ and show $C\cdot \hat B'(z\mu+\eta)= 0\mbox{ mod }(z-i|\eta|)$ has no nontrivial solution.
This is equivalent to $\det \hat B'(z\mu+\eta)|_{z=i|\eta|}\neq 0$. Let $\xi_1=z$ and $\xi_2=\eta_2,\xi_3=\eta_3$ in $\hat B$ with $(\eta_1,\eta_2)\neq 0$:
\begin{equation*}
\begin{split}
\hat B(z\mu+\eta)=
\begin{bmatrix}
-2z^3-5z|\eta|^2&-(z^2+|\eta|^2)\eta_2+3z^2\eta_2&-(z^2+|\eta|^2)\eta_3+3z^2\eta_3\\
2\eta_2(2|\eta|^2-z^2)&2z(-\eta_2^2+2z^2+2\eta_3^2)&-6z\eta_2\eta_3\\
2\eta_3(2|\eta|^2-z^2)&-6z\eta_2\eta_3&2z(-\eta_3^2+2z^2+2\eta_2^2)
\end{bmatrix}.
\end{split}
\end{equation*}
So its derivative is given by
\begin{equation*}
\begin{split}
\hat B'(z\mu+\eta)=
\begin{bmatrix}
-6z^2-5|\eta|^2&4z\eta_2&4z\eta_3\\
-4z\eta_2&-2\eta_2^2+12z^2+4\eta_3^2&-6\eta_2\eta_3\\
-4z\eta_3&-6\eta_2\eta_3&-2\eta_3^2+12z^2+4\eta_2^2
\end{bmatrix}.
\end{split}
\end{equation*}
Plug in $z=i|\eta|$, $z^2=-|\eta|^2$ to obtain
$\det \hat B'(i|\eta|\mu+\eta)=240|\eta|^6$.
Obviously it is never zero if $\eta\neq 0$. This completes the proof that $P_1$ in \eqref{EY} is an elliptic operator.

\medskip

Next we prove ellipticity for the operator $P_2$ given in \eqref{EH}. Recall that this operator is constructed by combining the linearized Einstein tensor with gauge terms. So we can apply the results and methods in \cite{Az2} on the ellipticity of stationary Einstein field equations.
It is shown in \cite{Az2} that the following operator 
\begin{equation}\label{E0}
\begin{split}
&L_0:[S^{m,\alpha}_{\delta}\times C^{m,\alpha}_{\delta}](M)\to [S^{m-2,\alpha}_{\delta+2}\times C^{m-2,\alpha}_{\delta+2}](M)\\
& \quad\quad\quad L_0( h, v)=
\big(
\tfrac{1}{2}D^*D h,~
\Delta v 
\big)\\
&B_0:[S^{m,\alpha}_{\delta}\times C^{m,\alpha}_{\delta}](M)\to [(\wedge_1)^{m-1,\alpha}\times S^{m,\alpha}\times C^{m,\a}](\p M)\\
&\quad\quad\quad B_0(h,v)=\big(
\beta h,~
h^T-2vg^T,~
H'_h-2\mathbf n(v)
\big)
\end{split}
\end{equation}
is elliptic. Here the first component of $L_0$ is the leading order term of $Ric'_h+\delta^{\ast}\beta h$, which is the linearized Ricci tensor with a gauge term. Note that the operator above is obtained from a conformal transformation of a boundary value problem of Einstein field equations in the projection formalism of stationary spacetimes (cf. section 2 of \cite{Az2}).
So we first define the conformal transformation
\begin{align*}
Q:[S^{m,\alpha}_{\d}\times C^{m,\alpha}_{\d}](M)\to [S^{m,\alpha}_{\d}\times C^{m,\alpha}_{\d}](M),\quad
Q(h, v)=( h-2vg, v).
\end{align*}
Notice that $Q$ is a linear isomorphism. So operator $P_2$ is elliptic if and only if the operator $P_3=P_2\circ Q$ is elliptic.  Composing $Q$ with the operator in \eqref{EH}, we obtain $P_3=(L_3,B_3)$ given by
\begin{equation*}
\begin{split}
&\quad\quad L_3:[S^{m,\alpha}_{\delta}\times C^{m,\alpha}_{\delta}](M)\to [S^{m-2,\alpha}_{\delta+2}\times C^{m-2,\alpha}_{\delta+2}](M)\\
& L_3(h, v)=
\big(
{\bf E}'_{ h}+\delta^*\delta  h-(\delta\delta  h)g+(\D v)g-D^2 v,~
\Delta \tr h+\delta\delta  h-4\D v
\big),\\
&B_3:[S^{m,\alpha}_{\delta}\times C^{m,\alpha}_{\delta}](M)\to [(\wedge_1)^{m-1,\alpha}\times S^{m,\alpha}\times C^{m,\a}](\p M)\\
&\quad\quad\quad B_3( h, v)=\big(
\delta h-d v,~
 h^T-2 vg^T,~
H'_{ h}-2\mathbf n( v)+v H_g
\big).
\end{split}
\end{equation*}
Here we throw away the terms involving only lower order ($\leq 1$) derivatives of $(h,v)$. We use ${\bf E}'_{h}$ to denote leading order part of the linearized Einstein tensor $Ein'_{h}$, i.e. 
$${\bf E}'_{h}=\tfrac{1}{2}D^*D h-\d^*\b h-\tfrac{1}{2}(\D \tr h+\d\d h)g.$$
Take trace of the first term in $L_3(h,v)$, multiply it by 2, add it to the second term of $L_3(h,v)$. Then we obtain the following operator $\w P=(\w L,\w B)$ which behaves the same as $P_3$ above regarding to ellipticity:
\begin{equation}\label{E1}
\begin{split}
&\quad\quad \w L:[S^{m,\alpha}_{\delta}\times C^{m,\alpha}_{\delta}](M)\to [S^{m-2,\alpha}_{\delta+2}\times C^{m-2,\alpha}_{\delta+2}](M)\\
&\w L( h, v)=
\big(
{\bf E}'_{ h}+\delta^*\delta  h-(\delta\delta  h)g+(\D v)g-D^2  v,~
4\D v-8\delta\delta  h
\big).\\
&\w B:[S^{m,\alpha}_{\delta}\times C^{m,\alpha}_{\delta}](M)\to [(\wedge_1)^{m-1,\alpha}\times S^{m,\alpha}\times C^{m,\a}](\p M)\\
&\quad\quad\quad \w B( h, v)=\big(
\delta h-d  v,~
 h^T-2 vg^T,~
H'_{ h}-2\mathbf n( v)
\big).
\end{split}
\end{equation}
The formal adjoint of $\w P$ is given by $\bar P=(\bar L,\bar B)$
\begin{equation}\label{E2}
\begin{split}
&\quad\quad \bar L:[S^{m,\alpha}_{\delta}\times C^{m,\alpha}_{\delta}](M)\to [S^{m-2,\alpha}_{\delta+2}\times C^{m-2,\alpha}_{\delta+2}](M)\\
& \bar L( h, v)=
\big(
{\bf E}'_h+\delta^*\delta h-D^2\tr h-8D^2v,~
4\Delta_{g} v -\delta\delta h+\Delta \tr h 
\big),\\
&\quad\quad\bar B:[S^{m,\alpha}_{\delta}\times C^{m,\alpha}_{\delta}](M)\to [(\wedge_1)^{m-1,\alpha}\times S^{m,\alpha}\times C^{m,\a}](\p M)\\
&\quad\quad \bar B( h, v)=\big(
\delta h-d\tr h-8dv,~
h^T+2vg^T,~
H'_h+2\mathbf n(v)
\big).
\end{split}
\end{equation}
It is straightforward to verify the adjointness via direct integration by parts. In the following we prove $\bar P$ and $\w P$ are adjoint operators using variations of the Einstein-Hilbert functional. Consider the functional 
$$I(g)=\int_M R_g~d\vol_g-2\int_{\p M}H_g~d\vol_{g^T}-16\pi m(g).$$
Here $R_g$ denotes the scalar curvature of $g$; and $m(g)$ denotes the ADM energy given by
$$
m(g)=\tfrac{1}{16\pi}\lim_{R\to\infty}\int_{S_R}(\p_i g_{ij}-\p_jg_{ii})\nu^j d\sigma
$$
where $S_R$ denotes the sphere $\{ r=R\}$ in $M$, $\nu$ denotes the unit normal to the sphere pointing to infinity with respect to the flat metric $\mathring g$ on $M$; and $d\sigma$ is the volume form on $S_R$ induced by $\mathring g$. Recall that both $ r$ and $\mathring g$ are obtained from the chart $M\cong\mathbb R^3\setminus B$. It is proved by Bartnik (cf.\cite{B2}) that the ADM energy is independent of the choice of this chart as long as $g$ is asymptotically flat with the decay rate $\delta>\tfrac{1}{2}$. In the functional above, the energy term is used to balance the boundary terms at infinity when we consider the variation of $I(g)$.

Given a metric $g\in Met_\d^{m,\alpha}(M)$ and deformations $h,k\in S^{m,\alpha}_\d(M)$, take a two-parameter family of metrics $g(t,s)=g+th+sk$.
Then the first variation of $I$ is given by (cf. for example \cite{AK})
\begin{equation*}
\begin{split}
I'_g(h)=\tfrac{\p}{\p t}|_{t=0,s=0}I(g(t,s))
&=\int_M -\<Ein_g,h\>+\int_{\p M}\<-H_gg^T+A_g,h^T\>.\\
\end{split}
\end{equation*}
Here and in the following we omit the volume forms $d\vol_g$ and $d\vol_{g^T}$. In the above, $A_g$ denotes the second fundamental form of the boundary $\p M\subset (M,g)$, i.e. $A_g(X,Y)=g(\nabla_X{\bf n},Y)$ for $X,Y\in T\p M$. In the following we use $A'_k$ to denote its the variation at $g$ with respect to deformation $k$. Basic calculation yields
$$A'_k=\tfrac{1}{2}\nabla_{\bf n}h+\d^*(h({\bf n})^T)+O_0(k).$$
Take the second variation of $I$:
\begin{equation*}
\begin{split}
I''_g(h, k)&=\tfrac{\p^2}{\p s\p t}|_{t=0,s=0}I(g(t,s))\\
&=\int_M -\<Ein'_{ k},h\>+\<Ein_g,h\circ k\>-\tfrac{1}{2}\tr k\<Ein_g,h\>\\
&+\int_{\p M}\<-H'_{ k}g^T-H_g k^T+A'_{ k},h^T\>-\<-H_gg^T+A_g,h^T\circ k^T\>+\tfrac{1}{2}\tr^T k\<-H_gg^T+A_g,h^T\>,
\end{split}
\end{equation*}
where $h\circ k=h_{ik} k^{k}_j+h_{jk} k^{k}_i$.
Repeat the calculation above to compute $I''(k,h)=\tfrac{\p^2}{\p t\p s}|_{t=0,s=0}I(g(t,s))$. Then by symmetry of the second variation we obtain
\begin{equation*}
\begin{split}
&\int_M -\<Ein'_{ k},h\>-\tfrac{1}{2}\tr k\<Ein_g,h\>+\int_{\p M}\<-H'_{ k}g^T+A'_{ k},h^T\>+\tfrac{1}{2}\tr^T k\<A_g,h^T\>\\
=&\int_M -\<Ein'_{h}, k\>-\tfrac{1}{2}\tr h\<Ein_g, k\>+\int_{\p M}\<-H'_{ h}g^T+A'_{h}, k^T\>+\tfrac{1}{2}\tr^Th\<A_g, k^T\>.
\end{split}
\end{equation*}
Now assume that $(h,v),(k,w)\in [S^{m,\a}_{\d}\times C^{m,\a}_{\d}](M)$ are two pairs of deformations  such that $\w B( k, w)=0$ and $\bar B(h,v)=0$.
Since $\w B_{ k}( k, w)=0$, we have $ k^T=2wg^T,~\tr^T k=4w,$ and $H'_{ k}=2{\bf n}( w)$ on $\p M$. Similarly, based on $\bar B(h,v)=0$ we have $ h^T=-2vg^T,~\tr^Th=-4v$ and $H'_h=-2{\bf n}(v)$ on $\p M$. We can plug these boundary conditions into the integral identity above. In particular, since $H_g=\tr^TA_g$ we have $H'_k=\tr^TA'_k-\<A_g,k^T\>$. So the term $\<A'_k,h^T\>$ can be simply computed as $\<A'_k,h^T\>=-2v\tr^TA'_k=-2vH'_k-2v\<A_g,k^T\>=-4v{\bf n}(w)-4vwH_g$. The same can be applied for the term $\<A'_h,k\>$. After simplification we obtain 
\begin{equation*}
\begin{split}
\int_M \<{Ein}'_{ k},h\>-\tfrac{1}{2}\tr h\<Ein_g, k\>=\int_M \<{Ein}'_{h}, k\>-\tfrac{1}{2}\tr k\<Ein_g,h\>+\int_{\p M}4[ v{\bf n}( w)-w {\bf n}(v)].
\end{split}
\end{equation*}
Note the integrand on the left side can be simplified as 
$\<{Ein}'_{ k}-\tfrac{1}{2}\<Ric_g, k\>g,h\>-\tfrac{1}{4}R_g\tr h\tr k$;
and $\<{Ein}'_{ k}-\tfrac{1}{2}\<Ric_g, k\>g,h\>
=\<Ric'_k-\tfrac{1}{2}(\D\tr k+\d\d k)g,h\>=\<{\bf E}'_k,h\>+O_0(k,h)$, where $O_0(k,h)$ involves only zero order derivatives of $k,h$ and is symmetric in $k,h$. We can do the same modification for the bulk integral on the right side.
Therefore the equation above can be further simplified as
\begin{equation}\label{ibp1}
\begin{split}
\int_M \<{\bf E}'_{ k},h\>= \int_M\<{\bf E}'_{h}, k\>+\int_{\p M}4[ v{\bf n}( w)-w {\bf n}(v)].
\end{split}
\end{equation}
Basic calculation of integration by parts on the remaining terms in $\w L(k,w)$ and $\bar L(h,v)$ yields
\begin{equation*}
\begin{split}
&\int_M\<\delta^*\delta  k-(\delta\delta  k)g+(\D w)g-D^2  w,h\>+\<4\D w-8\delta\delta  k,v\>\\
=&\int_M\<\delta^*\delta h-D^2 \tr h-8D^2 v,k\>+\<4\D v-\delta\delta h+\D(\tr h), w\>+\int_{\p M}{\bf B}[( k, w),(h,v)]
\end{split}
\end{equation*}
where in the boundary integral $\bf B$ is a bilinear form given by
\begin{equation*}
\begin{split}
{\bf B}[( k, w),(h,v)]=&-h(\d  k,{\bf n})+ k(\d h,{\bf n})- k({\bf n},d(\tr h))-(\tr h)\d k({\bf n})+(\tr h){\bf n}( w)- w{\bf n}(\tr h)\\
&+h({\bf n},d w)+ w\d h({\bf n})+4v{\bf n}( w)-4 w{\bf n}(v)-8 k({\bf n},d v)-8 v\d  k({\bf n}).
\end{split}
\end{equation*}
Since $\w B( k, w)=0$ and $\bar B(h,v)=0$, we have $\d k=d w$ and $\d h=d\tr h+8dv$ on $\p M$. Plugging these equalities into the expression above we obtain ${\bf B}[( k, w),(h,v)]=4[ -v{\bf n}( w)+w{\bf n}( v)]$. Combining this with \eqref{ibp1} we obtain that for all deformations $( k, w),(h,v)\in [S^{m,\a}_{\d}\times C^{m,\a}_{\d}](M)$ such that $\w B( k, w)=0$ and $\bar B(h,v)=0$:
\begin{equation*}
\int_M\<\w L( k, w),(h,v)\>=\int_M\<\bar L(h,v),(k,w)\>,
\end{equation*}
which justifies the adjointness between $\w P$ and $\bar P$.
\medskip

Now to prove that \eqref{EH} is elliptic it suffices to prove that both the operator $\w P$ and its adjoint operator $\bar P$ admit a uniform elliptic estimate (c.f.\cite{H,T}). In the following we apply the method in \cite{AK} to prove the elliptic estimate for $\w P$:
\begin{equation}\label{EP}
\begin{split}
||(h,v)||_{C^{m,\a}(M)}\leq C(||\w L(h,v)||_{C^{m-2,\a}(M)}+\textstyle\sum_i||\w B^{(i)}(h,v)||_{C^{m-k_i,\a}(\p M)}+||(h,v)||_{C^0(M)}),
\end{split}
\end{equation}
where $\w B^{(i)}(h,v)$ denotes the $i$th $(i=1,2,3)$ component of $\w B(h,v)$ expressed in \eqref{E1}; and the order $k_i$ equals to highest order of derivatives involved in the component.
Note that here and in the estimates to follow, it is sufficient to consider $M$ as a compact manifold with nonempty boundary. In fact, if the elliptic estimate above holds on compact manifold for both $\w P$ and $\bar P$, then according to \cite{H} the operator $\w P$ must be elliptic in the sense that $\w P$ has both left and right parametrix. Then according to \cite{B2,LM} it follows that the operator is also elliptic when defined on space of tensor fields with decay rate $1>\d>1/2$ on the noncompact manifold. 

Since the matrix of principal symbol of the interior operator $\w L$ in \eqref{E1} is very complicated to analyze directly, we first pair the boundary operator $\w B$ in \eqref{E2} with the simple operator \eqref{E0} and observe that they form an elliptic boundary value problem.
In fact since the principal symbol of $L_0$ is simply a rescaling of the identity matrix, ellipticity of $(L_0,\w B)$ can be immediately verified by checking the matrix of principal symbol of $\w B(\xi)$
\begin{equation*}
\w B(\xi)=
\resizebox{.4\textwidth}{!}{$
\begin{bmatrix}
-i\xi_1&-i\xi_2&-i\xi_3&0&0&0&-i\xi_1\\
0&-i\xi_1&0&-i\xi_2&-i\xi_3&0&-i\xi_2\\
0&0&-i\xi_1&0&-i\xi_2&-i\xi_3&-i\xi_3\\
0&0&0&1&0&0&-2\\
0&0&0&0&1&0&0\\
0&0&0&0&0&1&-2\\
0&-i\xi_2&0-i\xi_3&\tfrac{1}{2}i\xi_1&0&\tfrac{1}{2}i\xi_1&-2i\xi_1
\end{bmatrix}
$}
\end{equation*}
 is a non-degenerate matrix when $\xi=i|\eta|\mu+\eta$ for $\eta\neq 0$.
Thus we have the following elliptic estimate 
\begin{equation*}
\begin{split}
||(h,v)||_{C^{m,\a}(M)}\leq C(||L_0(h,v)||_{C^{m-2,\a}(M)}+\textstyle\sum_i||\w B^{(i)}(h,v)||_{C^{m-k_i,\a}(\p M)}+||(h,v)||_{C^0(M)}).
\end{split}
\end{equation*}
The interior operator $\w L$ and $L_0$ differ by 
\begin{equation*}
\begin{split}
\w L(h,v)-L_0(h,v)=
\big(
-\tfrac{1}{2}(\D\tr h+3\delta\delta  h)g-\tfrac{1}{2}D^2\tr h+(\D v)g-D^2  v,~
3\D v-8\delta\delta  h
\big).
\end{split}
\end{equation*}
So estimate \eqref{EP} will hold if the $C^{m-2,\a}$-norm of $\delta\delta h$, and $C^{m,\a}$-norm of $\tr h, v$ can be controlled by $\w P(h,v)$ in the sense that
\begin{equation}\label{dh}
\begin{split}
||\d\d h||_{C^{m-2,\a}(M)}\leq C(||\w L(h,v)||_{C^{m-2,\a}(M)}+\textstyle\sum_i||\w B_i(h,v)||_{C^{m-k_i,\a}(\p M)}+||(h,v)||_{C^0(M)})
\end{split}
\end{equation}
and the same for $||\tr h||_{C^{m,\a}(M)}, ||v||_{C^{m,\a}(M)}$.
Taking the divergence of the first component of $\w L(h,v)$ in \eqref{E2} we get:
\begin{equation*}
\begin{split}
\d({\bf E}'_{ h})+\d[\delta^*\delta  h-(\delta\delta  h)g+(\D v)g-D^2  v]=(\d\delta^*+d\d)(\delta  h-dv)+O_1(h).
\end{split}
\end{equation*}
In the equality above, we use fact that ${\bf E}'_h=Ein'_h+O_0(h)$ and that the Bianchi identity $\d Ein_g =0$ implies $\d(Ein'_h)=\d'_hEin_g=O_1(h)$.
The expression above can be understood as the elliptic operator $\delta\delta^*+d\d$ acting on the term $(\delta h -d v)$ whose Dirichlet boundary data is given in the first component of $\w B(h,v)$. Thus the $C^{m-1,\alpha}_{\d+1}$-norm of $(\delta h -d v)$ is controlled by
$||\d h-dv||_{C^{m-1,\a}(M)}\leq C(||\d\w L(h,v)||_{C^{m-3,\a}(M)}+||\w B^{(1)}(h,v)||_{C^{m-1,\a}(\p M)}+||(h,v)||_{C^{m-1,\a}(M)})$. 
Note here the first term $||\d\w L(h,v)||_{C^{m-3,\a}(M)}\leq ||\w L(h,v)||_{C^{m-2,\a}(M)}$; and the last term $||h||_{C^{m-1,\a}(M)}$ can be ignored according to the interpolation inequality. 
As a consequence, the $C^{m-2,\a}$-norm of $\delta^*(\delta h-d v)$ is controlled by $\w P(h,v)$.

Since the first component of $\w L(h,v)$ is the summation of ${\bf E}'_h$, $\d^*(\d h-dv)$ and $\tr[\d^*(\d h-dv)]g$, it follows immediately that the $C^{m-2,\a}$-norm of ${\bf E}'_h$ is also controlled $\w P(h,v)$. 
In addition, comparing the trace $\tr[ \d^*(\d h-dv)]$ with the second component of $\w L(h,v)$, we see that the $C^{m-2,\a}$-norm of $\Delta v$ and $\delta\delta h$ are controlled. 
In addition $\D\tr h=-2\tr{\bf E}'_h-\d\d h$. So we obtain estimate \eqref{dh} of $\d\d h$ and similar estimates for $\Delta \tr h$ and $\D v$.

It remains to control the $C^{m,\a}$-norm of $\tr h$ and $v$. Since we have obtained estimates for $\Delta \tr h$ and $\D v$, it suffices to prove certain boundary data of $\tr h, v$ has well-controlled norm.
To obtain such boundary data, we first apply the Gauss equation at $\partial M$ given by $|A|^2-H^2+R_{g^T}=R_g-2Ric_g(\mathbf n,\mathbf n)=-2Ein_g(\mathbf n,\mathbf n)$.
Its linearization is 
$$(|A|^2-H^2+R_{g^T})'_h=-2Ein'_h(\mathbf n,\mathbf n)-4Ein_g(\mathbf n'_h,\mathbf n)$$
where $A'_h$, $H'_h$ and $\mathbf n'_h$ only involve 1st and 0th order behavior of $h$, which can be ignored according to the interpolation inequality. So we obtain
$$\Delta_{g^T}\tr^Th^T+\delta^T\delta^Th^T=(R_{g^T})'_{h^T}+O_1(h)=-2{\bf E}'_h(\mathbf n,\mathbf n)+O_1(h).$$ 
The second boundary term in $\w B$ is $\w B^{(2)}(h,v)=h^T-2vg^T$, so $h^T=\w B_2+2vg^T$. Plug this into the equation above
$$2\Delta_{g^T}v=2{\bf E}'_h(\mathbf n,\mathbf n)-\Delta_{g^T}\tr^T\w B^{(2)}-\delta^T\delta^TB^{(2)}+O_1(h).$$
Recall that $||{\bf E}'_h||_{C^{m-2,\a}(M)}$ is already controlled by $\w P(h,v)$.
Since $\D_{g^T}$ is an elliptic operator on the boundary manifold, $||v||_{C^{m,\alpha}(\p M)}$ is also controlled by the operator $\w P(h,v)$ and so is the Dirichlet data $||h^T||_{C^{m,\a}(\p M)}$. Combining estimates for $||\Delta v||_{C^{m-2,\a}(M)}$ and Dirichlet data $||v||_{C^{m,\alpha}(\p M)}$, we obtain control of $||v||_{C^{m,\alpha}(M)}$. 

Extra boundary data for $\tr h$ is necessary to obtain control of its $C^{m-2,\a}$-norm. By the formula of variation of mean curvature in \eqref{lhn} we have 
\begin{equation}\label{G1}
\mathbf n(\tr h)=2H'_h-\delta^T(h(\mathbf n)^T)-(\delta h)(\mathbf n)+O_0(h).
\end{equation}
In the equation above, since $\d h-dv$ and $v$ are already controlled, $\delta h(\mathbf n)$ is controlled.
In addition, basic computation yields $(\delta h)^T=-\nabla_{\mathbf n}h(\mathbf n)^T+\delta^T(h^T)+O_0(h)$ inside which $||h^T||_{C^{m,\a}(\p M)}$ is controlled.
So we obtain control of $||\nabla_{\mathbf n}h(\mathbf n)^T||_{C^{m-1,\a}(\p M)}$ and hence its tangential divergence
\begin{equation}\label{G2}
\delta^T[\nabla_{\mathbf n}h(\mathbf n)^T]=\nabla_{\mathbf n}[\delta^T(h(\mathbf n)^T)]+O_1(h)
\end{equation}
is also controlled. Combining \eqref{G1} and \eqref{G2}, one obtains:
\begin{equation}\label{bh}
\begin{split}
\mathbf n\mathbf n(\tr h)=2\mathbf n(H'_h)-\delta^T(\nabla_{\mathbf n}h(\mathbf n)^T-\mathbf n\big(\delta h(\mathbf n)\big)+O_1(h).
\end{split}
\end{equation}
Note $\mathbf n(H'_h)$ in the above is controlled based on Riccati equation ${\bf n}(H)+|A|^2=-Ric({\bf n},{\bf n})$, where $Ric'_h({\bf n},{\bf n})={\bf E}'_h({\bf n},{\bf n})+1/2(\D\tr h+\d\d h)g+O_0(h)$ is well-controlled according to the previous analysis.
So every term on the righthand side of \eqref{bh} is under control and thus $||2\mathbf n\mathbf n(\tr h)||_{C^{m-2,\a}(\p M)}$ is controlled by $\w P(h,v)$.
Finally since the boundary data $2\mathbf n\mathbf n(\tr h)$ is elliptic for the Laplace operator $\D\tr h$, we can conclude that $||\tr h||_{C^{m,\a}(M)}$ is controlled by $\w P(h,v)$. 
This completes the proof of the estimate \eqref{EP}.

We can carry out the same process as above and derive the uniform elliptic estimate for $\bar P$ (cf. appendix \S4.4).

\medskip

\begin{remark}
{\rm Different from the work of Bartnik \cite{B2,B3} where the functions and tensor fields belong to the weighted Sobolev spaces, we work with the weighted H\"older spaces in this paper. The main reason is that when taking trace of a function one loses an extra $\tfrac{1}{2}$ regularity $H^s(M)\to H^{s-1/2}(\p M)$ which makes it complicated to discuss the ellipticity of the Bartnik boundary data.
 }
\end{remark}

\section{Critical points of the ADM mass}
In this section, we adopt the definitions of the ADM mass and the Regge-Teitelboim Hamiltonian from \cite{B3} and prove the corresponding result on the critical points for the ADM mass on the constraint manifold of initial data sets with fixed Bartnik boundary data.

We use the same notation as in \cite{B3}. A tensor field $\xi=(\xi^0,\xi^i)$ consisting of a scalar field $\xi^0$ and a vector field $\xi^i$ on $M$ is called a {\it spacetime vector field}. Let $\cT_{\d}^{m,\a}(M)$ denote the asymptotically zero spacetime tangent bundle, i.e. $\cT_{\d}^{m,\a}(M)=[C^{m,\alpha}_{\delta}\times T^{m,\alpha}_{\delta}](M)$. Fix a constant 4-vector $\xi_{\infty}=(\xi_{\infty}^0,\xi_{\infty}^i)$ $(i=1,2,3)$ defined on $\mathbb R^3\setminus B^3$. Pull it back to $M$ and obtain a parallel spacetime vector field, still denoted as ${\xi}_{\infty}$, with respect to the metric $\mathring g$. A smooth spacetime vector field $\hat\xi_{\infty}=(\hat\xi_{\infty}^0,\hat\xi_{\infty}^i)$ on $M$ is called a $constant~translation~near~infinity$ representing $\xi_{\infty}$, if there is a $R$ such that $\hat\xi_{\infty}= \xi_{\infty}$ on $E_{2R}$, and $\hat\xi_{\infty}=0$ on $M\setminus E_{R}$, where $E_R=\{p\in M:~ r(p)> R\}$. Let $\mathcal Z^{m,\a}_{\d}(M)$ denote the space of {\it asymptotic translation vector fields}, i.e. 
\begin{equation*}
\begin{split}
\cZ^{m,\a}_{\d}(M)=\{\xi\in &[C^{m,\a}\times T^{m,\a}](M):\\
&\xi-\hat\xi_{\infty}\in \cT^{m,\a}_{\d}(M)\text{ for some constant translation near infinity }\hat\xi_{\infty}\}.
\end{split}
\end{equation*}
For convenience, we turn to the $(g,\pi,\w\Phi)$ formulation of the constraint map as described in \S2.
For $(u,Z)\in\cT$, let $\w{\mathcal C}_B(u,Z)$ be the level set
\begin{equation*}
\begin{split}
\w{\mathcal C}_B(u,Z)=\{(g,\pi)\in\w{\mathcal B}:\w\Phi(g,\pi)=(u,Z)\}
\end{split}
\end{equation*}
where $\mathcal{\w B}$ and $\w\Phi$ are defined as in \eqref{tB} and \eqref{tconsm}. Since the space $\w{\mathcal C}_B(u,Z)$ is equivalent to ${\mathcal C}_B(u,Z)$, it is also a smooth Banach manifold.

The general ADM total energy-momentum vector $\mathbb{P}$ is defined in \cite{B3} by describing its pairing with a constant vector $\xi_{\infty}\in\mathbb{R}^{1,3}$
\begin{equation*}
\begin{split}
&16\pi\xi_{\infty}^{0}\mathbb{P}_0(g,\pi)=\int_{M}\hat\xi_{\infty}^0\mathcal R_0(g)+\mathring{\nabla}^i\hat\xi^0_{\infty}(\mathring{\nabla}^jg_{ij}-\mathring{\nabla}_itr_{\mathring{g}}g)d\vol_{\mathring{g}}\\
&16\pi\xi_{\infty}^{i}\mathbb{P}_i(g,\pi)=2\int_{M}\big(\hat\xi_{\infty}^i\mathcal P_{0i}(\pi)+\pi^{ij}\mathring{\nabla}_i\hat\xi_{\infty j}\big)d\vol_{\mathring{g}}
\end{split}
\end{equation*}
inside which $\hat\xi_{\infty}$ is a representative translation vector at infinity for $\xi_{\infty}$ and 
\begin{equation*}
\begin{split}
\mathcal R_0(g)=\mathring{\nabla}^{ij}g_{ij}-\Delta_0tr_{\mathring g}g,\ \  \mathcal P_{0i}(\pi)=\mathring g\mathring{\nabla}_k\pi^{jk}.
\end{split}
\end{equation*}
It is easy to generalize the result in \cite{B3} to obtain that $\mathbb P$ defines a smooth function on the Banach manifold $\w{\mathcal C}_B(u,Z)$ under the condition that $(u,Z)\in [C^{k,\a}_{q}\times (\wedge_1)^{k,\a}_{q}](M)$ for some $q\geq 4$ and $k\geq 0$.
So in this section we always assume this condition holds
\footnote{We note that this condition can be replaced by the integrable condition that $u,Z_i$ are $L^1$ on $M$.}.
Moreover, in this case $\bP$ agrees with the usual formal definition of ADM energy-momentum vector and is independent of choice of the chart $M\cong\mathbb R^3\setminus B$.

We adopt the Regge-Teitelboim Hamiltonian defined in \cite{B3} to our setting:
\be\label{CH}
\begin{split}
\cH:\w{\cB}&\times\cZ^{m,\a}_{\d}(M)\to\bR\\
\mathcal H(g,\pi;\xi)=\int_M\langle (\hat\xi_{\infty}-\xi),\w\Phi(g,\pi)\rangle &
+\int_M\hat\xi_{\infty}^0(\mathcal R_0(g)-\w\Phi_0(g,\pi))+\int_M\mathring{\nabla}^i\hat\xi_{\infty}^0(\mathring{\nabla}^jg_{ij}-\mathring{\nabla}_itr_{\mathring{g}}g)\\
&+\int_M\hat\xi_{\infty}^i(\mathcal P_{0i}(\pi)-\w\Phi_i(g,\pi))+\int_M 2\pi^{ij}\mathring\nabla_i\hat\xi_{\infty,j},
\end{split}
\ee
inside which $\hat\xi_{\infty}$ is a constant translation at infinity such that $\xi-\hat\xi_{\infty}\in \cT^{m,\a}_{\d}(M)$. Here and in the following we omit the volume form $d\vol_{\mathring g}$.
Based on \cite{B3}, the functional $\mathcal H$ is smooth and bounded.
Moreover, on the constraint manifold $\w{\mathcal C}_{B}(u,Z)$
the functional can be equivalently expressed as
\begin{equation}\label{CH1}
\begin{split}
\mathcal H(g,\pi;\xi)=16\pi\xi^{\mu}_{\infty}\mathbb{P}_{\mu}-\int_M\xi^{\mu}\w\Phi_{\mu}(g,\pi),
\end{split}
\end{equation}
where $\xi_{\infty}$ is the constant vector equal to the asymptotic limit of $\xi$.
The following lemma describes the variation of $\mathcal H$.
\begin{lemma} If $\xi\in\cZ^{m,\a}_{\d}(M)$ then for all $(g,\pi)\in\w{\cB}$ and $(h,p)\in T\w{\cB}|_{(g,\pi)}$
\begin{equation}\label{LH}
\begin{split}
D_{(g,\pi)}\mathcal H(g,\pi;\xi)(h,p)=-\int_M(h,p)\cdot D\w \Phi_{(g,\pi)}^*\xi.
\end{split}
\end{equation}
\end{lemma}
\begin{proof}
Using integration by parts, we can write the linearization of the first term in \eqref{CH} as:
\begin{equation}\label{ibp}
\begin{split}
\int_M \langle(\hat\xi_{\infty}-\xi),D\w\Phi_{(g,\pi)}(h,\sigma)\rangle=&\int_M \langle D\w\Phi_{(g,\pi)}^*(\hat\xi_{\infty}-\xi),(h,\sigma)\rangle+\int_{\partial M}\w B[(\hat\xi_{\infty}-\xi),(h,\sigma)]+\lim_{r\to\infty}\int_{S_r}\w B,
\end{split}
\end{equation}
where in the boundary integral $\w B$ is a bilinear form given by (cf\cite{B3} equation (82)),
\begin{equation}\label{bi}
\begin{split}
\w B[(\mu,Y),(h,\sigma)]=&\mathbf n^{i}[\mu(\nabla^{j}h_{i j}-\nabla_{i}\tr h)-h_{i j}\nabla^{j}\mu+\tr h\nabla_{i}\mu]\sqrt{g}\\
&+2\mathbf n^{i}[Y_{j}\sigma_{i}^{ j}+Y^{j}\pi^{k}_{i}h_{j k}-\frac{1}{2}Y_{i}\pi^{j k}h_{j k}],
\end{split}
\end{equation}
with $(\mu,Y)=\hat\xi_{\infty}-\xi=-\xi$ on the boundary $\p M$. We also refer to equations (6)-(9) of \cite{B3} for the explicit formula of $D\w\Phi$ and its adjoint $D\w\Phi^*$.
According to the boundary conditions in \eqref{tB}, any deformation $(h,\sigma)\in T\mathcal{\w B}$ satisfies
\begin{equation}\label{BH}
\begin{split}
\begin{cases}
h_{AB}=0,\mbox{ for }A,B=2,3\\
\mathbf n(\tr^Th)+2\delta^T(h(\mathbf n)^T)-h_{11}H=0\\
\sigma^{11}+\frac{1}{2}\pi^{11}h_{11}=0\\
\sigma_{1A}+\pi_{11}h_{1A}=0,\mbox{ for }A=2,3
\end{cases}
\text{on } \partial M.
\end{split}
\end{equation}
Recall that the index $1$ denotes the normal direction to $\p M$ and indices $2,3$ denote the tangential direction. 
The second equation above implies that
$\mathbf n(\tr^Th)-2(\nabla^T)^Ah_{A1}-h_{11}H=0$ for $A=2,3$.
Basic calculation gives
$\mathbf n^i\nabla^jh_{ij}
=\mathbf n(h_{00})+(\nabla^T)^Ah_{A1}+h_{11}H$. Combining those two equalities we can derive that
\begin{equation*}
\begin{split}
\mathbf n^{i}(\nabla^{j}h_{i j}-\nabla_{i}\tr h)=-\mathbf n(\tr^Th)+(\nabla^T)^Ah_{A1}+h_{11}H
=-(\nabla^T)^Ah_{A1}.
\end{split}
\end{equation*}
In addition,
\begin{equation*}
\begin{split}
\mathbf n^{i}[-h_{i j}\nabla^{j}\mu+\tr h\nabla_{i}\mu]
=-h_{1A}\nabla^A\mu-h_{00}\mathbf n(\mu)+h_{00}\mathbf n(\mu)=-h_{1A}\nabla^A\mu
\end{split}
\end{equation*}
where we use the fact that $h_{AB}=0$ from \eqref{BH}. Summing up the two equations above we obtain that the first line in \eqref{bi} can be written as
\begin{equation*}
\begin{split}
\mathbf n^{i}[\mu(\nabla^{j}h_{i j}-\nabla_{i}\tr h)-h_{i j}\nabla^{j}\mu+\tr h\nabla_{i}\mu]
=-\mu(\nabla^T)^Ah_{A1}-h_{1A}\nabla^Au=-(\nabla^T)^A(\mu h_{A1}),
\end{split}
\end{equation*}
which is a pure divergence term on the boundary $\partial M$ and hence its integral is zero. As for the second line in \eqref{bi}, we have 
\begin{equation*}
\begin{split}
&\mathbf n_{i}[Y_{j}\sigma^{ij}+Y^{j}\pi^{ki}h_{j k}-\frac{1}{2}Y^{i}\pi^{jk}h_{j k}]
=Y_1\sigma^{11}+Y_A\sigma^{A1}+Y^{j}\pi^{k1}h_{jk}-\frac{1}{2}Y_1\pi^{jk}h_{jk}\\
=&-\frac{1}{2}Y^1\pi^{11}h_{11}-Y^A\pi^{11}h_{1A}+Y^A\pi^{11}h_{1A}+Y^1\pi^{k1}h_{1k}
-\frac{1}{2}Y_1\pi^{11}h_{11}-Y_1\pi^{1A}h_{1A}=0.
\end{split}
\end{equation*}
In the second equality above, we use the last two equations in \eqref{BH} to replace $\sigma$ with $\pi,h$ and the first equation in \eqref{BH} to throw away terms involving $h_{AB}~(A,B=2,3)$. Thus the boundary integral over $\partial M$ in \eqref{ibp} must vanish. The integral at infinity is also zero because $\xi-\hat\xi_{\infty}$ and $(h,\sigma)$ decay fast enough to zero. Thus we obtain
\begin{equation*}
\begin{split}
\int_M \langle(\hat\xi_{\infty}-\xi),D\w\Phi_{(g,\pi)}(h,\sigma)\rangle=\int_M \langle D\w\Phi_{(g,\pi)}^*(\hat\xi_{\infty}-\xi),(h,\sigma)\rangle.
\end{split}
\end{equation*}
The linearization of the remaining terms in \eqref{CH} is given by $-\int_M\langle (h,\sigma),D\w\Phi_{(g,\pi)}^*(\hat\xi_{\infty})\rangle$  (cf.\cite{B3} Theorem 5.2). Combining this with the equation above we obtain \eqref{LH}.

\end{proof}

When the energy-momentum vector $\bP(g,\pi)$ is time-like, the ADM total mass of the initial data $(g,\pi)$ is defined as
\begin{equation}\label{ADM}
\begin{split}
m_{\rm ADM}(g,\pi)=\sqrt{-\mathbb{P}^{\mu}\mathbb{P}_{\mu}}.
\end{split}
\end{equation}
With the lemma above we can now prove that any critical point of the ADM total mass on the constraint manifold $\w \cC_{B}(u,Z)$ must admit a generalised Killing vector field, which is the analog of Corollary 6.2 in \cite{B3}. A spacetime vector field $\xi\in [C^{m,\a}\times T^{m,\a}](M)$ is called {\it a generalised Killing vector field} of the initial data set $(M,g,\pi)$ (or $(M,g,K)$) if $D\w\Phi^*|_{(g,\pi)}(\xi)=0$ (or $D\w\Phi^*|_{(g,K)}(\xi)=0$). In this case, the initial data set $(M,g,\pi)$ is called a {\it generalised stationary initial data set}.

\begin{theorem} Suppose $(u,Z)\in[C^{k,\a}_{q}\times(\wedge_1)^{k,\a}_{q}](M)~(k\geq 0,q\geq 4)$, $(g_0,\pi_0)\in\w\cC_B(u,Z)$ and $\mathring\bP=\bP(g_0,\pi_0)$ is a time-like vector. If $(Dm_{\rm ADM})_{(g_0,\pi_0)}(h,\sigma)=0$ for all $(h,\sigma)\in T\w\cC_B(u,Z)|_{(g_0,\pi_0)}$ then $(g_0,\pi_0)$ admits a generalised Killing vector field $\xi$ which has a limit at infinity proportional to
$\mathring\bP$. Conversely, if $(g_0,\pi_0)$ is a generalised stationary initial data set, then $(Dm_{\rm ADM})_{(g_0,\pi_0)}(h,p)=0$ for all $(h,p)\in T\cC_B(u,Z)|_{(g_0,\pi_0)}$.
\end{theorem}
\begin{proof} If $(g_0,\pi_0)$ is a critical point with time-like ADM total energy-momentum vector $\mathring\bP$,
let $(\xi_0)_{\infty}$ be the constant vector $(\xi_0)^{\mu}_{\infty}= -\mathring\bP^{\mu}/m_{\rm ADM}(g_0,\pi_0)\in\mathbb R^4$ and define a functional on $E$ on $\w \cC_B(u,Z)$ by  
$$E(g,\pi)=(\xi_0)_{\infty}^{\mu}\mathbb P_{\mu}(g,\pi).$$ 
Clearly, the derivative of the ADM mass at $(g_0,\pi_0)$ is 
$Dm_{\rm ADM}|_{(g_0,\pi_0)}=-(m_{\rm ADM})^{-1/2}\mathring\bP^{\mu}D\bP_{\mu}=(\xi_0)^{\mu}_{\infty}D\mathbb P_{\mu}=DE|_{(g_0,\pi_0)}$. 
So $(g_0,\pi_0)$ is also a critical point of $E$ on the constraint manifold.
Let $(\hat\xi_0)_{\infty}$ be a constant translation near infinity representing $(\xi_0)_{\infty}$. Choose $\xi_0\in\cZ^{m,\a}_{\d}(M)$ such that $\xi_0-(\hat\xi_0)_{\infty}\in\cT^{m,\a}_{\d}(M)$
and plug it into the formula \eqref{CH1} to obtain a functional on $(g,\pi)$:
\begin{equation*}
\begin{split}
\mathcal H(g,\pi;\xi_0)=16\pi(\xi_0)^{\mu}_{\infty}\mathbb{P}_{\mu}-\int_M\xi_0^{\mu}\w\Phi_{\mu}(g,\pi).
\end{split}
\end{equation*} 
Observe on the constraint manifold $\w\cC_B(u,Z)$, $(g_0,\pi_0)$ is a critical point of the first term on the right side and the second term is constant. Thus we obtain
$$D_{(g_0,\pi_0)}\mathcal H{(g,\pi;\xi_0)}(h,\sigma)=0\quad\mbox{ for all }(h,\sigma)\in T\w \cC_{B}(u,Z).$$ 
By a Lagrange-multiplier argument (cf.\cite{B3} Theorem 6.3), there is $(\w X^0,\w X)\in(\mathcal T)^*$ such that for all $(h,\sigma)\in T\w{\mathcal B}|_{(g_0,\pi_0)}$:
\begin{equation*}
\begin{split}
(\w X^0,\w X)[D\w\Phi_{(g_0,\pi_0)}(h,\sigma)]=D_{(g_0,\pi_0)}\mathcal H{(g,\pi;\xi_0)}(h,\sigma)=\int_M\langle D\w\Phi_{(g_0,\pi_0)}^*(\xi_0),(h,\sigma)\rangle,
\end{split}
\end{equation*}
where we apply Lemma 3.1 on the right side.
It follows that $(\w X^0,\w X)$ is a weak solution of
$D\Phi^*_{(g_0,K_0)}(\w X^0,\w X)=D\w\Phi_{(g_0,K_0)}^*(\xi_0)$.
Here we use the equivalence between $(\cB,\Phi)$ and $(\w\cB,\w\Phi)$ and $(g_0,K_0)$ denotes the correspondence in $\cB$ of $(g_0,\pi_0)$. Since $\xi_0\in\cZ^{m,\a}_\d$, we can prove in the same way as in \S3.1 that $(\w X^0,\w X)$ is a regular solution, i.e. it is $C^{m,\alpha}$ smooth in ${\rm int}M$. 
Let $(X^0,X)=(\xi_0^0-\w X^0,\xi_0^i-\w X)$, then
for all compactly supported $(h,p)\in T{\mathcal B}|_{(g_0,K_0)}$
\begin{equation}\label{WX}
\begin{split}
\int_M\langle ( X^0, X),D\Phi_{(g_0,K_0)}(h,p)\rangle=0.
\end{split}
\end{equation}
Thus $(X^0,X)$ must behave as in \eqref{ax1} or \eqref{ax2} asymptotically.
Combining this with the fact that $(\w X^0,\w X)=\xi_0-(X^0,X)$ is a bounded linear functional on $\cT$, it is easy to derive that $(\w X^0,\w X)$ must be asymptotically zero at the rate of $\d$ (cf. \S4.5 for details).
In addition, $(X^0,X)$ is also $C^{m,\alpha}$ smooth up to the boundary (cf. appendix \S4.4). 
Therefore, $(X^0,X)$ is a generalised $C^{m,\a}$ Killing vector field on $M$ and $(X^0,X)-\xi_0\in C^m_{\delta}$, i.e. the limit of $(X^0,X)$ at infinity is proportional to the ADM energy-momentum vector of $(g_0,\pi_0)$.

Conversely, suppose $(g_0,\pi_0)$ admits a generalised Killing vector field $\hat X\in \cZ^{m,\a}_{\d}(M)$ whose asymptotic limit $(\hat X)_{\infty}$ is proportional to $\mathring\bP$. 
Based on Lemma 3.1, $D\Phi^*_{(g_0,K_0)}(\hat X)=0$ implies that $D_{(g_0,\pi_0)}\cH (g,\pi,\hat X)(h,\sigma)=0$ for all $(h,\sigma)\in T\w\cC_B(u,Z)|_{(g_0,\pi_0)}$. This further implies that $(\hat X)_{\infty}^{\mu}D\bP_{\mu}|_{(g_0,\pi_0)}=0$ on the constraint manifold. Since $(\hat X)_{\infty}$ is proportional to $\mathring\bP$, we also have $\mathring\bP^{\mu}D\bP_{\mu}|_{(g_0,\pi_0)}=0$ i.e. $(Dm_{\rm ADM})_{(g_0,\pi_0)}(h,\sigma)=0$ based on the discussion at the beginning of the proof.
\end{proof}

\section{Appendix}
In this section we provide the details which are left open in some proofs of this paper.
\subsection{Transform from $\mathcal B$ to $\mathcal{\w B}$}
Recall at the beginning of \S 2 we have defined the space $\mathcal B$ of pairs $(g,K)$ fixing the Bartnik boundary data and an equivalent space $\mathcal{\w B}$. In the following we verify that the boundary conditions in the tangent space $T\cB$ given in \eqref{lBcon} are equivalent to those of $T\mathcal{\w B}$ in \eqref{tB}. The reparametrization space $\mathcal{\w B}$ is equivalent to $\mathcal B$ via the map
\begin{equation*}
\begin{split}
&P:\mathcal {\w B}\to\mathcal B\\
P(g,\pi)=&(~g,~\big(\pi^{\flat}-\frac{1}{2}(tr_g\pi)g\big)/\sqrt{g}~).
\end{split}
\end{equation*}
Let $(h,\sigma)\in T\w\cB|_{(g,\pi)}$ be a infinitesimal deformation at $(g,\pi)$. Then linearization of $P$ is given by
\begin{equation*}
DP_{(g,\pi)}(h,\sigma)=\big(h,p(h,\sigma)\big),
\end{equation*}
where
\begin{equation*}
\begin{split}
p_{ij}(h,\sigma)=[\sigma_{ij}-\frac{1}{2}(\tr\sigma)g_{ij}+\pi^{k}_{i}h_{kj}+\pi^{k}_{j}h_{ki}-\frac{1}{2}(\tr\pi) h_{ij}-\frac{1}{2}(\pi^{k\tau}h_{k\tau})g_{ij}-\frac{1}{2}\tr h(\pi_{ij}-\frac{1}{2}\tr\pi g_{ij})]/\sqrt{g}.
\end{split}
\end{equation*}
Since $(h,p)\in T\mathcal B$, it must satisfy the boundary conditions listed in \eqref{lBcon}. 
It is obvious that the first two boundary conditions in \eqref{lBcon} and \eqref{tB} are the same.
The third condition in \eqref{lBcon} is equivalent to
\begin{equation*}
\begin{split}
0=\tr^Tp=&\tr^T\sigma-(\tr_g\sigma)
+2\pi^{A1}h_{A1}-\pi^{lk}h_{lk}
-\frac{1}{2}\tr h(\tr^T\pi-\tr\pi)
=-\sigma^{11}-\frac{1}{2}\pi^{11}h_{11},
\end{split}
\end{equation*}
where we use the condition $h^T=0$ on $\p M$. This gives the third boundary condition in \eqref{tB}.
Finally along $\p M$ we have 
\begin{equation*}
\begin{split}
&p({\bf n})^T
=[\sigma_{1A}+\frac{1}{2}h_{11}\pi_{1A}+\pi^B_Ah_{B1}+\pi_{11}h_{1 A}-\frac{1}{2}(\tr_{g}\pi) h_{1A}]/\sqrt{g}\\
&K(\mathbf n'_h)^T=\big(\pi^{\flat}(\mathbf n'_h)^T-\frac{1}{2}\tr_{g}\pi g(\mathbf n'_h)^T\big)/\sqrt{g}
=[-\pi_{BA}h^B_1-\frac{1}{2}h_{11}\pi_{1A}+\frac{1}{2}(\tr_g\pi) h_{1A}]/\sqrt{g},
\end{split}
\end{equation*}
where we apply $h^T=0$ on $\p M$ and the variation formula of the unit normal $\mathbf n'_{h}=-h_{1A}-\frac{1}{2}h_{11}\mathbf n$. Summing up the equations above, we derive that the last condition in \eqref{lBcon} can be equivalently expressed in terms of $(h,\sigma)$ as
$0=p(\mathbf n)^T+K(\mathbf n'_h)^T=
\big(\sigma_{1A}+\pi_{11}h_{1A}\big)/\sqrt{g}$, which is the last boundary condition listed in \eqref{tB}.
\subsection{Decomposition of constraints at the boundary} Given any $(u,Z)\in \cT$, we show that there exists some $(h,p)\in T\cB|_{(g,K)}$ such that $(u,Z)=D\Phi_{(g,K)} (h,p)$ on $\p M$. Then it follows naturally that any $(u,Z)\in\cT$ can be decomposed as $(u,Z)=(u_0,Z_0)+(u_1,Z_1)$ with $(u_0,Z_0)$ vanishing on $\p M$ and $(u_1,Z_1)=D\Phi_{(g,K)}(h,p)\in{\rm Im}\cL$, which is the decomposition used at the end of the proof of surjectivity in \S2.1.

For simplicity we can first choose $(h,p)$ so that $h$ vanishes to the first order on $\p M$ and $p$ vanishes to the zero order on $\p M$, i.e.
\begin{equation}\label{42}
h_{ij}=0,~{\bf n}(h_{ij})=0,~p_{ij}=0\text{ on }\p M.
\end{equation}
Obviously $(h,p)\in T\cB|_{(g,K)}$. Based on \eqref{lu} the linearization $(D\Phi_0)_{(g,K)}(h,p)$ is given by:
\begin{equation*}
\begin{split}
(D\Phi_0)_{(g,K)}(h,p)=&(-\Delta (\tr h)+\delta\delta h)\sqrt{g}=-\mathbf n\big(\mathbf n(\tr^Th)\big)\sqrt{g}\mbox{ on }\p M.
\end{split}
\end{equation*} 
Here we use the equality that $\Delta (\tr h)=\Delta^T (\tr h) +\mathbf n(\tr h)H-\mathbf n\mathbf n(\tr h)$ and $\delta\delta h=\nabla^i\nabla^jh_{ji}=\mathbf n\mathbf n(h_{11})$ on the boundary. Set $h$ satisfying \eqref{42} and such that 
$$-{\bf n}({\bf n}(\tr^Th))\sqrt{g}=u\text{ on }\p M,$$
then we have $(D\Phi_0)_{(g,K)}(h,p)=u$. 
Next plug $(h,p)$ into \eqref{lz} and obtain 
\begin{equation*}
\begin{split}
(D\Phi_i)_{(g,K)}(h,p)=-2\big(\delta p+d\tr p\big)\sqrt{g}=-2{\bf n}(\tr^Tp)\sqrt{g}\cdot {\bf n}-2\nabla_{\bf n}p({\bf n})^T\sqrt{g}\mbox{ on }\p M.
\end{split}
\end{equation*}
Thus we can choos $p$ satisfying \eqref{42} and such that ${\bf n}(\tr^Tp)=Z_1$ and ${\bf n}(p_{1A})=Z_A,~A=2,3$ on $\p M$. It then follows that $(D\Phi_i)_{(g,K)}(h,p)=Z$.  
\subsection{Construction of the extension maps} We provide possible approach to construct the maps $E_1$ and $E_2$, which are used in the proof of splitting kernel in \S2.2.

Fix $\tau\in{\rm Ker}\d^T$ on the boundary $\p M$ of the Riemannian manifold $(M,g)$. We can first fix a collar neighborhood $U$ of the boundary $\p M$ inside which the flow of the distance function to the boundary is well-defined. Without loss of generality, assume $U=[0,1)\times\p M$ where we use $s\in[0,1)$ to denote the function of distance to the boundary. We extend the unit normal vector ${\bf n}$ naturally to be the unit vector field perpendicular to the $s$-level set in $U$ pointing to the infinity, i.e. ${\bf n}=\p_s$. Define $h$ so that $h^T=0,~h(\p_s,\p_s)=0$ in $U$ and $h=0$ on $\p M$. Here the superscript $^T$ denotes the component of a tensor tangential to the level set $\{s={\rm constant}\}$ in $U$. It follows immediately that $H'_h=0$ and ${\bf n}'_h=0$ on $\p M$ based on the formula \eqref{lhn}. The 1-form $h({\bf n})^T$ is defined to be such that 
\begin{equation*}
h({\p_s})^T=0\ \ {\rm on}~\p M,\ \ 
\nabla_{\p_s}(h({\p_s})^T)=-\tau\ \ {\rm in}~U,
\end{equation*} 
where we think of $\tau$ as being Lie-dragged by $\p_s$ in $U$, i.e. $L_{\p_s}\tau=0$.
Then we have $(\d h)^T_A
=-\nabla^1h_{1A}-\nabla^Bh_{BA}
=-\nabla_{\bf n}(h({\bf n})^T)=\tau_A$ on $\p M$.
Next fix a smooth jump function $f(s)$ in $U$ such that $f(s)=1$ for $0\leq s\leq 1/4$ and $f(s)=0$ for $s\geq 1/2$. We can now define $E_1(\tau)=f(s)h$ where $h$ is as constructed above and $fh$ is extended trivially as a symmetric (0,2)-tensor defined on $M$ vanishing outside $U$. Then it is easy to check that $E_1$ is a linear bounded map. 

Next fix $h\in S^{m,\a}_{\d}(M)$ on the initial data set $(M,g,K)$. Take the collar neighborhood $U$ as above inside which the vector field ${\bf n}'_h$ is well-defined. Construct a symmetric 2-tensor $\w h$ in $U$ such that $\w h^T=\w h({\bf n},{\bf n})=0$ and $\w h({\bf n})^T=-K({\bf n}'_h)^T$ in $U$. Then extend $\w h$ to a global tensor field, still denoted as $\w h$, which is equal to zero outside $U$ and equal to $f(s)\w h$ in $U$.  Now define $E_2(h)=\w h$ as constructed. It follows that $E_2$ is a linear bounded map. Since ${\bf n}'_h$ involves only 0-order data of $h$ as shown in \eqref{lhn}, so is the map $E_2$ .
\subsection{Ellipticity of the adjoint operator $\bar P$} In the following we prove the uniform elliptic estimate for the operator $\bar P$ defined in \eqref{E2}, using the same idea as in the proof for $\w P$. We first try to pair $\bar B$ with an interior operator with simpler principal symbol. However, it is easy to check the operator $L_0$ in \eqref{E0} which is paired with $\w B$ in \S2.3 does not work for $\bar B$. So we modify $L_0$ as 
\be\label{BL0}
\begin{split}
\bar L_0(h,v)=(D^*Dh,~2\d\d h+ \D v),
\end{split}
\ee
and show that $(\bar L_0,\bar B)$ is elliptic.
The symbol of $\bar L_0$ is given by
\begin{equation*}
\begin{split}
\bar L_0(\xi)=
\begin{bmatrix}
|\xi|^2{\bf I}_{6\times 6}&0\\
{\bf v}(\xi)& |\xi|^2
\end{bmatrix},
\mbox{ with }
{\bf v}(\xi)=
2\begin{bmatrix}
-\xi_1^2 & -\xi_1\xi_2 & -\xi_1\xi_3 & -\xi_2^2 & -\xi_2\xi_3 & -\xi_3^2
\end{bmatrix}.
\end{split}
\end{equation*}
Obviously the determinant is given by $\ell_0=|\xi|^{14}$ with $\ell_0^+=(z-i|\eta|)^7$. So the interior operator is properly elliptic. The adjoint matrix is given by
\begin{equation*}
\begin{split}
\bar L^*_0(\xi)=
\begin{bmatrix}
|\xi|^{12}{\bf I}_{6\times 6}&0\\
-|\xi|^{10}{\bf v}(\xi)&\quad |\xi|^{12}
\end{bmatrix}
=|\xi|^{10}\cdot
\begin{bmatrix}
|\xi|^2{\bf I}_{6\times 6}&0\\
-{\bf v}(\xi)& |\xi|^2
\end{bmatrix}.
\end{split}
\end{equation*}
So to prove the complementing boundary condition we need to show that there is no nonzero complex vector $C$ such that
\begin{equation*}
\begin{split}
C\cdot\tfrac{1}{|\xi|^{10}}\bar B(z\mu+\eta)\bar L^*_0(z\mu+\eta)=C\cdot \bar B(z\mu+\eta)\cdot
\begin{bmatrix}
(z^2+|\eta|^2){\bf I}_{6\times 6}&0\\
-{\bf v}(z\nu+\eta)& (z^2+|\eta|^2)
\end{bmatrix}
=0\mbox{ mod }(z-i|\eta|)^2.
\end{split}
\end{equation*}
The matrix of principal symbol of the boundary operator is given by
\begin{equation*}
\begin{split}
\bar B(z\mu+\eta)=
\resizebox{.45\textwidth}{!}{$
\begin{bmatrix}
-2iz&-i\eta_2&-i\eta_3&-iz&&-iz&-8iz\\
-i\eta_2&-iz&&-2i\eta_2&-i\eta_3&-i\eta_2&-8i\eta_2\\
-i\eta_3&&-iz&-i\eta_3&-i\eta_2&-2i\eta_3&-8i\eta_3\\
&&&1&&&2\\
&&&&1&&\\
&&&&&1&2\\
&-i\eta_2&-i\eta_3&\tfrac{1}{2}iz&&\tfrac{1}{2}iz&2iz
\end{bmatrix}.
$}
\end{split}
\end{equation*}
To simplify the computation, when checking the linear relation of rows of the product matrix $\bar B\cdot\bar L_0^*$, we can take invertible row operations on $\bar B$ before taking the matrix product. The following matrix is obtained by a series of row operations on $\bar B$, which can be summarized as a left product by an invertible matrix:
\begin{equation*}
\begin{split}
\hat B(z\mu+\eta)=
\resizebox{.35\textwidth}{!}{$
\begin{bmatrix}
-i&&&&&&&-4i\\
&i&&&-2\eta_2&-\eta_2&-2\eta_2&\\
&&&i&-2\eta_3&-\eta_3&-2\eta_3&\\
&&&&1&&&\\
&&&&&1&&\\
&&&&&&1&\\
\tfrac{1}{2}i&&&&&&&i
\end{bmatrix}
$}
\cdot B(z\mu+\eta)
=
\resizebox{.35\textwidth}{!}{$
\begin{bmatrix}
-2z&-5\eta_2&-5\eta_3&z&&z&\\
\eta_2&z&&&&-\eta_2&\\
\eta_3&&z&-\eta_3&&&\\
&&&1&&&2\\
&&&&1&&\\
&&&&&1&2\\
z&\tfrac{3}{2}\eta_2&\tfrac{3}{2}\eta_3&&&&2z
\end{bmatrix}.
$}
\end{split}
\end{equation*}
Now take the product $\tfrac{1}{|\xi|^{10}}\hat B\cdot\bar L_0^*={\bf M}$ we obtain
\begin{equation*}
\begin{split}
{\bf M}(z\mu+\eta)=
\resizebox{.85\textwidth}{!}{$
\begin{bmatrix}
-2z(z^2+|\eta|^2)&-5\eta_2(z^2+|\eta|^2)&-5\eta_3(z^2+|\eta|^2)&z(z^2+|\eta|^2)&&z(z^2+|\eta|^2)&\\
\eta_2(z^2+|\eta|^2)&z(z^2+|\eta|^2)&&&&-\eta_2(z^2+|\eta|^2)&\\
\eta_3(z^2+|\eta|^2)&&z(z^2+|\eta|^2)&-\eta_3(z^2+|\eta|^2)&&&\\
4z^2&4z\eta_2&4z\eta_3&4\eta_2^2+(z^2+|\eta|^2)&4\eta_2\eta_3&4\eta_3^2&2(z^2+|\eta|^2)\\
&&&&(z^2+|\eta|^2)&&\\
4z^2&4z\eta_2&4z\eta_3&4\eta_2^2&4\eta_2\eta_3&4\eta_3^2+(z^2+|\eta|^2)&2(z^2+|\eta|^2)\\
4z^3+z(z^2+|\eta|^2)&4z^2\eta_2+\tfrac{3}{2}\eta_2(z^2+|\eta|^2)&4z^2\eta_3+\tfrac{3}{2}\eta_3(z^2+|\eta|^2)&4z\eta_2^2&4z\eta_2\eta_3&4z\eta_3^2&2z(z^2+|\eta|^2)
\end{bmatrix}.
$}
\end{split}
\end{equation*}
Obviously row 5 is zero mod $(z-i|\eta|)$. So we need to take the derivative ${\bf M}'(z)$ of matrix ${\bf M}$ with respect to $z$ and show there is no nonzero complex vector $C$ such that
$C{\bf M}'(z)=0\mbox{ mod }(z-i|\eta|^2).$
The derivative of ${\bf M}$ is given by
\begin{equation*}
\begin{split}
{\bf M}'(z)=
\resizebox{.7\textwidth}{!}{$
\begin{bmatrix}
-6 z^2-2|\eta|^2&-10\eta_2 z&-10\eta_3 z&3 z^2+|\eta|^2&&3 z^2+|\eta|^2&\\
2\eta_2 z&3 z^2+|\eta|^2&&&&-2\eta_2 z&\\
2\eta_3 z&&3 z^2+|\eta|^2&-2\eta_3 z&&&\\
8z&4\eta_2&4\eta_3&2 z&&&4z\\
&&&&2 z&&\\
8z&4\eta_2&4\eta_3&&&2 z&4z\\
15z^2+|\eta|^2&11z\eta_2&11z\eta_3&4\eta_2^2&4\eta_2\eta_3&4\eta_3^2&6z^2+2|\eta|^2
\end{bmatrix}.
$}
\end{split}
\end{equation*}
It is easy to check that $\det{\bf M}'(z=i|\eta|)\neq 0$.
Thus the complementing boundary condition holds. So $(\bar L_0,\bar B)$ is elliptic and hence so is $(\tfrac{1}{2}\bar L_0,\bar B)$. Therefore, we obtain the following elliptic estimate:
\begin{equation*}
\begin{split}
||(h,v)||_{C^{m,\a}}\leq C(||\tfrac{1}{2}\bar L_0(h,v)||_{C^{m-2,\a}}+||\bar B^{(i)}(h,v)||_{C^{m-k_i,\a}}+||(h,v)||_{C^0}).
\end{split}
\end{equation*}
The interior operator $\bar L$ and $L_0$ differ by 
\begin{equation*}
\begin{split}
\bar L(h,v)-\tfrac{1}{2}\bar L_0(h,v)=
\big(
-\frac{1}{2}(\Delta \tr h+\delta\delta h)g-\frac{3}{2}D^2\tr h-8D^2v,~
\tfrac{7}{2}\D v-2\delta\delta h+\D(\tr h)
\big).
\end{split}
\end{equation*}
So elliptic estimate for $\bar P$ will hold if we can control $||\delta\delta h||_{C^{m-2,\a}(M)},~||\tr h||_{C^{m,\a}(M)}$ and $||v||_{C^{m,\a}(M)}$ by $\bar P(h,v)$ similarly as in \eqref{dh}.
Taking the divergence of the first term of $\bar L(h,v)$ we get:
\begin{equation*}
\begin{split}
\delta[\bar L^{(1)}(h,v)]=\delta\delta^*(\delta h -d\tr h-8d v)
\end{split}
\end{equation*}
inside which we use the Bianchi identity as in the proof for $\w P$.
Note the expression above can be taken as $\delta\delta^*$ -- an elliptic operator -- acting on the term $(\delta h -d\tr h-8d v)$ whose Dirichlet boundary data is included in $\bar B(h,v)$. Thus $(\delta h -d\tr h-8d v)$ is controlled by $\bar P^*(h,v)$ as well as ${\bf E}'_h$.
So we get control of $(\delta\delta h-\Delta \tr h-8\Delta v)$. Combining this with the trace $\tr {\bf E}'_h=-\tfrac{1}{2}(\D\tr h+\d\d h)$ and second component of $\bar L(h,v)$, we see that $||\Delta v||_{C^{m-2,\a}(M)}$ and $||\delta\delta h||_{C^{m-2,\a}(M)}$ and $||\Delta\tr h||_{C^{m-2,\a}(M)}$ are all controlled by $\bar P$. 
Then we can derive control of $||v||_{C^{m,\alpha}(M)}$ and $||h^T||_{C^{m,\alpha}(\p M)}$ by analyzing the Gauss equation at $\partial M$, in the same way as for $\w P$. So it remains to obtain boundary condition for $\tr h$. By the formula of variation of mean curvature we have 
$$\mathbf n(\tr h)=2H'_h-\delta^T(h(\mathbf n)^T)-(\delta h)(\mathbf n)+O_0(h).$$
In the equation above, $\delta h(\mathbf n)$ is not controlled, but we have control of the boundary data $(\delta h-d\tr h-8dv)$ with $dv$ is already controlled.  So we can rewrite the equation above as:
\begin{equation}\label{G10}
2\mathbf n(\tr h)=2H'_h-\delta^T(h(\mathbf n)^T)-[(\delta h-d(\tr h)](\mathbf n)+O_0(h).
\end{equation}
In addition, basic computation yields $-(\delta h-d\tr h)^T=\nabla_{\mathbf n}h(\mathbf n)^T+\delta^T(h^T)+\nabla^T\tr h+O$ inside which $\d h-d\tr h$ and $h^T$ are both controlled on $\p M$.
So we get control of $\nabla_{\mathbf n}h(\mathbf n)^T+\nabla^T\tr h$ on $\partial M$ and hence also its tangential divergence
\begin{equation}\label{G20}
\delta^T[\nabla_{\mathbf n}h(\mathbf n)^T+\nabla^T\tr h]=\nabla_{\mathbf n}[\delta^T(h(\mathbf n)^T)]+\Delta_{g^T}\tr h.
\end{equation}
Combining \eqref{G10} and \eqref{G20}:
\begin{equation*}
\begin{split}
2\mathbf n\mathbf n(\tr h)-\Delta_{g^T}\tr h=2\mathbf n(H'_h)-[\delta^T(\nabla_{\mathbf n}h(\mathbf n)^T+\nabla^T\tr h)]-\mathbf n[(\delta h-d(\tr h))(\mathbf n)]+O_1(h).
\end{split}
\end{equation*}
Now every term on the righthand side of the equation above is under control.
Finally recall that $\Delta \tr h$ is controlled by $\bar P$ and it is elliptic when combined with the boundary term $2\mathbf n\mathbf n(\tr h)-\Delta_{g^T}\tr h$. Therefore, $\tr h$ is also controlled by $\bar P$. This completes the proof of elliptic estimate for $\bar P$.
\subsection{Boundary behavior of the generalised Killing vector field} In the following we show that a generalised Killing vector field $(X^0,X)$ is $C^{m,\a}$ smooth up to the boundary $\p M$, which is used at the end of the proof of Theorem 3.2. Suppose $(X^0,X)$ is a weak solution of $D\Phi_{(g,K)}^*(X^0,X)=0$. So it is $C^{m,\alpha}$ in the interior ${\rm int}M$ of $M$ and satisfies 
\be\label{A1}
\begin{split}
\begin{cases}
2X^0K+L_Xg=0\\
D^2X^0+L_XK+X^0\big(-Ric_g+2K\circ K-(trK)K+\frac{1}{4}ug\big)=0
\end{cases}
\mbox{ in int}M.
\end{split}
\ee
Recall that $(g,K)\in\cB$ and $\Phi(g,K)=(u,Z)$.
We can apply the approach in \cite{BC} here to show that $(X^0,X)$ is $C^{m,\a}$ smooth up to the boundary $\p M$. From the equations above we can obtain
\be\label{A2}
\begin{split}
&\nabla_i\nabla_jX^0=-L_XK_{ij}+X^0\big(-Ric_g+2K\circ K-(trK)K+\frac{1}{4}ug\big)_{ij}\\
&\nabla_i\nabla_jX_k=R_{kjmi}X^m+D_k\big(X^0K_{ij}\big)-D_i\big(X^0K_{jk}\big)-D_j\big(X^0K_{ki}\big),
\end{split}
\ee
Here we use $R_{kjmi}$ to denote the curvature tensor of $g$. The second equation is obtained by taking convariant derivative of the first equation in \eqref{A1} and applying the Bianchi identity $R_{kijm}+R_{ijkm}+R_{jkim}=0$ with $R_{kimj}X^m=\nabla_k\nabla_iX_j-\nabla_i\nabla_kX_j$.
Recall that $ r$ is the radius function on $M$ obtained by pull-back from $\mathbb R^3\setminus B$. So $ r\in [1,+\infty)$ and $ r=1$ on $\partial M$. Let $s=1- r$. So $s\in(-\infty,0]$, and $\p_s=\tfrac{x^i}{s-1}\p_i$ in the Cartesian coordinates. We consider the limit $\lim_{s\to 0^-}(X^0,X)$. Derivatives of $(X^0,X)$ are given by
\begin{align*}
&{\partial_s}X^0=\frac{x^i}{s-1}\partial_i X^0,\ \ \partial_s X_i=\frac{x^j}{s-1}(\nabla_jX_i+\Gamma_{ji}^kX_k),\\
&\partial_s\nabla_iX^0=\frac{x^j}{s-1}(\nabla_j\nabla_iX^0+\Gamma_{ji}^k\nabla_kX^0),\ \
\partial_s\nabla_iX_j=\frac{x^k}{s-1}(\nabla_k\nabla_iX_j+\Gamma_{ki}^l\nabla_lX_j+\Gamma_{kj}^l\nabla_iX_l).
\end{align*}
Here $X_i=g_{ik}X^k$; and the covariant derivative $\nabla$ and Christoffel symbol $\Gamma_{ij}^k$ are with respect to the metric $g$. Notice that the terms $\nabla_k\nabla_iX_j$ and $\nabla_k\nabla_iX_j$ can be replaced by lower derivatives based on \eqref{A2}. So define $f=(X^0,\nabla X^0,X,\nabla X)$ and let $F=|f|^2$. The equations above imply that 
$|\frac{\partial F}{\partial s}|\leq CF\text{ for some constant }C>0$. Then integration yields
\begin{equation*}
\begin{split}
\partial_s(e^{-Cs}F)\leq 0\Rightarrow e^{-Cs_1}F(s_1)\leq e^{-Cs_2}F(s_2)~\forall s_1\geq s_2.
\end{split}
\end{equation*}
Thus $F(s)\leq e^{Cs+C}F(-1)+C~\mbox{ for all }0>s>-1$. So $\lim_{s\to 0^-}F(s)\leq C$. Take an open neighborhood of the boundary $\partial M\subset U\subset M$. Then $X^0,\nabla X^0,X,\nabla X$ are uniformly bounded in $U\setminus\partial M$. Henceforth, $\nabla_i\nabla_jX^0$ and $\nabla_i\nabla_jX$ are also uniformly bounded in $U\setminus \p M$ according to \eqref{A2}. This in return implies $\nabla X^0,\nabla X$ are uniformly continuous in $U\setminus \p M$
and so is $(X^0,X)$. Based on \eqref{A2} again, $\nabla^2X^0,\nabla^2X$ are also uniformly continuous. Therefore, we can extend $(X^0,X)$ to be well-defined and second order differentiable in $U$. In fact by a bootstrap argument it can be extended to a $C^{m,\alpha}$ fields in $U$ and by continuity we also have $D\Phi^*|_{(g,K)}(X^0,X)=0$ up to the boundary.
\subsection{Aymptotic behavior of Killing vector field} Finally, we provide a detailed discussion on the asymptotical behavior of elements which belong to the dual space $(\cT)^*$ and also the kernel of the adjoint $D\Phi^*$.  Here we work with the spacetime vector $(\w X^0,\w X)$ in the proof of Theorem 3.2. The same discussion works well for the element $\hat X$ at the end of proof of surjectivity in \S2.1. 

It is shown in the proof of Theorem 3.2 that $(X^0,X)$ is a $C^{m,\a}$ solution of $D\Phi^*(X^0,X)=0$. So according to Proposition 2.1 of \cite{BC}, its asymptotic behavior must be as in \eqref{ax1} or \eqref{ax2}. Suppose the spacetime vector $\w X=\xi_0-X$ is not asymptotically zero. Then there exist constants $\Lambda_{\mu\nu}$ not all zero such that
\begin{equation}\label{ax41}
\begin{split}
\w X^i=-(\xi_0)_{\infty}^i+\Lambda_{ij}x^j+O_m(r^{1-\delta}),\quad \w X^0=-(\xi_0)_{\infty}^0+\Lambda_{0i}x^i+O_m(r^{1-\delta});
\end{split}
\end{equation}
or constants $A^{\mu}$ with $A^{\nu}-(\xi_0)_{\infty}^{\nu}$ not all zero such that 
\begin{equation}\label{ax42}
\begin{split}
\w X^i=-(\xi_0)_{\infty}^i+A^i+O_m(r^{-\delta}),\quad \w X^0=-(\xi_0)_{\infty}^0+A^0+O_m(r^{-\delta}).
\end{split}
\end{equation}
Here we adopt the notation in \cite{BC} that $f=O_m(r^{-\d})$ if there exists a constant $C$ such that 
$$|\p_{x_1}^{i_1}\p_{x_2}^{i_2}...\p_{x_n}^{i_n}f|\leq Cr^{-\d-|i|}\mbox{ for all } |i|= i_1+i_2+...+i_n\in\{0,1,..,m\}.$$
This is equivalent to that $f\in C^m_\d(M)$.

Since both $X$ and $\xi_0$ are $C^{m,\a}$ smooth on $M$, so is every component $\w X^{\mu},~(\mu=0,1,2,3)$ of $\w X$.
In addition, $\w X^{\mu}$ is a bounded functional on $C^{m-2,\alpha}_{\delta+2}(M)$, so for any function $v\in C^{m-2,\alpha}_{\delta+2}(M)$ the $L^2$-pairing
$\int_{M}\w X^{\a}\cdot vd\vol_g$ must be finite.
If $\w X$ behaves as in \eqref{ax41}, without loss of generality we can assume $\wedge_{01}\neq 0$. Let $v=v(r)$ be the smooth positive function which equals to zero near the interior boundary and equals to $\frac{\Lambda_{0i}x^i}{r^{\beta+3}}~(\beta>\delta)$ for $r>R$. Then $v\in C^{m-2,\alpha}_{\delta+2}(M)$ and 
\begin{equation*}
\begin{split}
\int_{r>R}\w X^0\cdot vd\vol_g
&=\int_R^{\infty}\int_{S^2}\frac{\Lambda_{0i}x^i}{r^{\beta+3}}\w X^0r^2ds^2dr
=\int_R^{\infty}\int_{S^2}\frac{\Lambda_{0i}x^i}{r^{\beta+1}}\big(-(\xi_0)_{\infty}^0+\Lambda_{0i}x^i+O(r^{1-\delta})\big)ds^2dr\\
&=\int_R^{\infty}\int_{S^2}\frac{r^2}{r^{\beta+1}}(\wedge_{0i}\tfrac{x^i}{r})^2ds^2dr-\int_R^{\infty}\int_{S^2}\frac{\wedge_{0i}x^i}{r^{\beta+1}}(\xi_0)_{\infty}^0+\int_R^{\infty}\int_{S^2}O(r^{1-\delta-\beta})ds^2dr.
\end{split}
\end{equation*}
Obviously if $1<\beta\leq 2$, the above integral diverges which contradicts that $\w X^0$ is a bounded functional. 

If $\w X$ behaves as in \eqref{ax42}, again without loss of generality we assume $A^0-(\xi_0)_{\infty}^0\neq 0$. Let $v=v(r)$ be a smooth positive function which equals to zero near the interior boundary and equals to $\frac{1}{r^{\beta+2}}~(\beta>\delta)$ for $r>R$. Then 
\begin{equation*}
\begin{split}
\int_{r>R}\w X^0\cdot vd\vol_g
&=\int_R^{\infty}\int_{S^2}\frac{1}{r^{\beta+2}}\w X^0r^2ds^2dr
=\int_R^{\infty}\int_{S^2}\frac{1}{r^{\beta}}(-(\xi_0)_{\infty}^0+A^0+O(r^{-\delta}))ds^2dr\\
&=(-(\xi_0)_{\infty}^0+A^0)\int_R^{\infty}\int_{S^2}\frac{1}{r^{\beta}}ds^2dr+\int_R^{\infty}\int_{S^2}O(r^{-\delta-\beta})ds^2dr
\end{split}
\end{equation*}
For $\text{max}\{\delta,1-\delta\}<\beta<1$, the above integral diverges which also yields a contradiction. So $\w X$ must decay to zero asymptotically.

\medskip

\bibliographystyle{plain}

\end{document}